\documentclass[12pt]{amsart}
\usepackage[utf8]{inputenc}
\usepackage[initials]{amsrefs}
\usepackage[
        colorlinks,
]{hyperref}

\usepackage{fullpage,amssymb,amsmath,dsfont}

\newcommand{\Gal}{{\rm Gal}}
\newcommand{\Ord}{{\rm ord}}
\newcommand{\nr}{\mathrm{unr}}
\newcommand{\GL}{{\rm GL}}
\newcommand{\Id}{{\rm Id}}

\newcommand{\Q}{\mathbb Q}

\newcommand{\N}{\mathbb N}
\newcommand{\C}{\mathbb C}
\newcommand{\Z}{\mathbb Z}

\newtheorem{X}{X}[section]
\newtheorem{cor}[X]{Corollary}
\newtheorem{lem}[X]{Lemma}
\newtheorem{prop}[X]{Proposition}
\newtheorem{thm}[X]{Theorem}

\theoremstyle{definition}
\newtheorem{defi}[X]{Definition}

\theoremstyle{remark}
\newtheorem{rk}{Remark}[section]
\newtheorem*{exe}{Example}

\begin{document}

\title{On the influence of the Galois group structure\\ on the Chebyshev bias in number fields}
\author{Mounir Hayani}
\date{}
\maketitle

\begin{abstract}
    In this paper we produce unconditionally new instances of Galois number field extensions exhibiting strong discrepancies in the distribution of Frobenius elements among conjugacy classes of the Galois group. We first prove an inverse Galois theoretic statement showing a dichotomy between ``extreme Chebyshev biases'' and ``equal prime ideal counting''. We further introduce a group theoretic property that implies extreme biases. In the case of abelian extensions this leads to a complete characterization of Galois groups enabling extreme biases. In the case where the Galois group is a $p$-group, a simple criterion is deduced for the existence of extreme biases, and associated effective statements of Linnik type are obtained.
\end{abstract}

\section{Introduction and statement of results}\label{Mainresults}

\subsection{Background on Chebyshev's bias}\label{subsec:cheb}

In 1896, de La Vallée Poussin proved that for every $q \geq 2$, prime numbers are uniformly distributed among invertible residue classes modulo $q$. More precisely, denoting $$\pi(x;q,a)=\#\{p\  \text{prime}\colon p \leq x,\ p\equiv a \text{ mod }q\}\qquad(x\geq 2),$$
the Prime Number Theorem in arithmetic progressions asserts that

\[
\pi(x;q,a) \sim \frac{1}{\varphi(q)}\frac{x}{\log(x)} \qquad(x\to\infty)\,.
\]

Half a century before, examining the specific case of $q=4$, Chebyshev sought to conduct a more detailed comparison between prime numbers $p\equiv 1 \text{ mod }4$ and prime numbers $p\equiv 3 \text{ mod }4$. In a letter sent to Fuss \cite{Chebyshev}, Chebyshev noted that for a given $x\geq 2$, the quantity $\pi(x;4,3)$ seems to be greater  than $\pi(x;4,1)$. Since then, this discrepancy has been referred to as \emph{Chebyshev's Bias}.

For any $q\geq 3$, Rubinstein and Sarnak studied in~\cite{RS94} the general question of Chebyshev's bias between invertible classes modulo $q$. Under the generalized Riemann hypothesis  Rubinstein and Sarnak proved, for $q\geq 3$, and $a\ne b$ two invertible classes modulo $q$, the existence of the \emph{logarithmic density}
$$\delta(q;a,b):= \underset{X \rightarrow +\infty}{\lim}\frac{1}{\log(X)} \int_{2}^{X}  \mathds 1_{\pi(x;q,a)>\pi(x;q,b)}(t) \frac{\mathrm{d}t}{t}\,.$$ 
Additionally assuming a linear independence assumption on non-trivial zeros of Dirichlet $L$-functions, they confirm Chebyshev's original observation of the existence of a bias towards $a$ (that is $\delta(q;a,b)>\frac 12$) if and only if $a$ is not a quadratic residue mod $q$ and $b$ is a quadratic residue mod $q$.

\medskip
Following the ideas of Rubinstein and Sarnak, Ng \cite{Ng} studied Chebyshev's bias in the context of number fields. 
Let $L/K$ be a Galois extension of number fields, with group $G$. Consider $t\colon G\rightarrow \C$ a class function on $G$. Define the prime ideal counting function for any $x\geq 2$:
\[
\pi(x;L/K;t):=\underset{\substack{\frak{p}\lhd O_K \\ N\frak{p}\leq x}} {\sum}t(\varphi_{\frak{p}})\,.
\]
Here the sum runs over non-zero prime ideals $\mathfrak p$ of $O_K$ and $\varphi_\mathfrak p$ denotes the corresponding Frobenius conjugacy class defined up to inertia. In order to unify the treatment of ramified and unramified prime ideals we use the standard definition (\emph{e.g.} used in~\cite{FJ}*{page 28}) valid for any $k\geq1$: 
$$t(\varphi_{\frak{p}}^k):=\frac{\sum_{s \in I_{\frak{P/p}}}  t(\varphi_{\frak{p}}^k s)}{|I_{\frak{P/p}}|}\,,$$
where $\frak{P}$ is a prime of $O_L$ lying above $\frak{p}$ and $I_{\frak{P/p}}$ denotes the inertia group attached to $\frak{P/p}$.

For two conjugacy classes $C_1,C_2$ of $G$ we define the class function on $G$ : 
\begin{equation}\label{def:C1C2}
t_{C_1,C_2}=\frac{|G|}{|C_1|}\mathds 1_{C_1}-\frac{|G|}{|C_2|}\mathds{1}_{C_2}\,,
\end{equation}
where $\mathds 1_{C_i}$ is the indicator function of $C_i$.
Denoting $\mathcal{P}_{C_1,C_2}:=\{x\geq 2\colon \pi \left(x;L/K;t_{C_1,C_2}\right)>0\}$, the logarithmic density of $\mathcal{P}_{C_1,C_2}$ (when it exists) is the generalization to Galois extensions of number fields of $\delta(q;a,b)$, that is $$\delta(\mathcal{P}_{C_1,C_2}):=\underset{X \to +\infty}{\lim}\frac{1}{\log X}\int_2^{X} \mathds 1_{\mathcal{P}_{C_1,C_2}}(x)\frac{\mathrm{d}x}{x}\,.$$
In his Thesis \cite{Ng}, Ng studied densities of type $\delta(\mathcal{P}_{C_1,C_2})$. Assuming Artin's holomorphy conjecture, GRH, as well as a linear independence assumption on the set of zeros of Artin $L$-functions, Ng proved in \cite{Ng} the existence of $\delta(\mathcal{P}_{C_1,C_2})$ and computed its value in several explicit cases. Since then, the study of $\delta(\mathcal{P}_{C_1,C_2})$ has seen numerous developments (see \emph{e.g.}~\cite{Dev,Bail,FJ}). Devin (\cite{Dev}*{Th. 2.5}) establishes conditionally to assumptions of the same type (although substantially weaker than the hypotheses made in~\cite{Ng}) that $0<\delta(\mathcal{P}_{C_1,C_2})<1$. In particular the expected (but unproved) diophantine properties of Artin $L$-function zeros prevent extreme biases occuring (see Definition~\ref{def:extbias}).

In this paper we will study in the context of normal number field extensions, settings where the assumptions of~\cite{Dev}*{Th. 2.5} are not met, thereby producing the following phenomenon. 

\begin{defi}[Extreme Chebyshev bias]\label{def:extbias}
We say that the Galois extension $L/K$ has an extreme Chebyshev bias relatively to $(C_1,C_2)$ where $C_1$, $C_2$ are two conjugacy classes of $\Gal(L/K)$, if  up to exchanging $C_1$ with $C_2$, one has $\pi(x;L/K;t_{C_1,C_2})>0$ for big enough $x$ hence in particular $\delta(\mathcal{P}_{C_1,C_2})$ exists and is equal to $1$.
\end{defi}

In \cite{Bail}*{\S 1.3}, Bailleul explains that Ng's linear independence (LI) assumption cannot hold in general for a relative extension of number fields $L/K$. Instead, Bailleul suggests in the case $K=\Q$ a refined version of LI under which he computes $\delta(\mathcal{P}_{C_1,C_2})$ for some families of normal extensions (with dihedral and quaternion Galois groups). In \cite{FJ}, Fiorilli and Jouve give a detailed probabilistic study of the fluctuations of the error term in the Chebotarev density Theorem within families of normal number field extensions and obtain general estimates for $\delta(\mathcal{P}_{C_1,C_2})$ that they apply (under assumptions analogous to those used in~\cite{Ng} or~\cite{Bail}) to various Galois group structures (Abelian, supersolvable, $\mathfrak S_n$,...).

In strong contrast with the framework of~\cite{Dev,Bail,FJ}, Fiorilli and Jouve (\cite{FJ2}) show unconditionally that there exists infinitely many Galois extensions of number fields having an extreme Chebyshev bias. More precisely they prove, for a tower $L/K/\Q$ of number fields such that $L/\Q$ is Galois with group $G^+$, that if two conjugacy classes $C_1$ and $C_2$ in $G=\Gal(L/K)$ \emph{do not} have the same number of square roots\footnote{One easily sees that the count of square roots is a class function on any group.} in $G$ but are contained in a common conjugacy class $C^+$ of $G^+$, then $L/K$ has an extreme Chebyshev bias with respect to $C_1$ and $C_2$ (see~\cite{FJ2}*{Th. 1.1}).
Natural questions that arise are the following:
\begin{enumerate}
    \item[\bf{a)}] Can we produce extreme Chebyshev biases in the case of conjugacy classes with \emph{the same} number of square roots and contained in a common conjugacy class of $G^+$ ?
    \item[\bf{b)}] Is it possible to classify the group structures enabling such 
    extreme biases (see~\cite{FJ2}*{\S 2, Remark})? 
\end{enumerate}
The main goal in this paper is to answer these questions. We now state our main results.

\subsection{Statement of the main results}

Our main contributions can be split into two different settings: first we let $k$ be an arbitrary number field and $G$ be an arbitrary finite group. Denoting by $\mathfrak S(G)$ the group of all permutations of elements of $G$, our first results are stated in the setting of the \emph{Cayley embedding} of $G$ into $\mathfrak S(G)$ (\emph{i.e.} we let $G$ act on itself by left translation giving rise to an injective group homomorphism $G\to\mathfrak S(G)$). Letting $L$ be a Galois extension of $k$ of group $G^+\simeq \mathfrak S(G)$ and denoting $K=L^G$, we obtain unconditional results on discrepancies for the distribution of Frobenius elements in conjugacy classes of $G=\Gal(L/K)$. The second framework for our main results aims at relaxing the assumption according to which $G\hookrightarrow G^+$ is the Cayley embedding. For example, we will produce an infinite family of abelian groups $G$ exhibiting an extreme Chebyshev bias that are Galois groups of an extension $L/K$ such that the degree of $K$ over $\Q$ is exactly $4$.
We will also produce extreme Chebyshev biases for subextensions of Galois extensions $L/k$ of prime power degree.

\medskip


We will use the following notation: given a group $G$ and $g\in G$, we denote by $\Ord(g)$ the order of $g$ in $G$. Moreover, if $D\subset G$ is a conjugacy invariant subset, we will write $\pi(x;L/K,D)$ as a shorthand for $\pi(x;L/K,\mathds 1_D)$.

\subsubsection{Cayley embedding and Chebyshev biases in number fields}

Our first main result uncovers the following dichotomy phenomenon. 

\begin{thm}\label{princ}
    Let $G$ be a finite group and let $k$ be a number field. Consider the injection $G \hookrightarrow \frak S (G)=:G^+$ (the group of permutations of $G$), given by the action of $G$ on itself by left translation. Let $L$ denote a Galois extension of $k$ with group $G^+$ and let $K=L^G$ be the subextension of $L/k$ fixed by $G$. Then, for all $a,b\in G$ such that $\Ord(a)=\Ord(b)$, with respective conjugacy classes $C_a$ and $C_b$, one of the following cases occurs: 
    \begin{enumerate}
        \item either for all $x\geq 2$: $$\frac{1}{|C_a|}\pi(x;L/K;C_a)=\frac{1}{|C_b|}\pi(x;L/K;C_b)\,,$$
        \item or there exists $A>0$ such that, up to exchanging $C_a$ and $C_b$, we have for all $x\geq A$, 
        \[ \frac{1}{|C_a|}\pi(x;L/K;C_a)>\frac{1}{|C_b|}\pi(x;L/K;C_b)\,.
        \]
        Thus $L/K$ has an extreme Chebyshev bias with respect to $C_a$ and $C_b$.
    \end{enumerate}
\end{thm}

When $G$ is not isomorphic to the symmetric groups $\frak{S}_1,\frak{S}_2,\frak{S}_3$, then $G$ always contains two elements $a$ and $b$ of the same order such that their respective conjugacy classes $C_a$ and $C_b$ are distinct (see \cite{Feit}*{Corollary B.2}). Thus, given any group $G$ of order at least $3$ and not isomorphic to $\mathfrak S_3$,  Theorem~\ref{princ} answers a non-trivial Chebyshev bias question in a specific Galois extension $L/K$ with Galois group $G$.

The proof of Theorem~\ref{princ} will be obtained as a consequence of a more general statement (Theorem~\ref{OC-thm}) which makes precise under which condition (1) (resp. (2)) of Theorem~\ref{princ} occurs. This distinction relies entirely on group theoretic properties of the conjugacy classes $C_1$ and $C_2$. The following statement gives an instance of both alternatives appearing in Theorem~\ref{princ}.

\begin{cor}\label{Example1}
    With notation as in Theorem  \ref{princ}, we have:
    \begin{enumerate}
        \item If $G=Q_8$ is the quaternion group of order $8$, then, for all $a,b \in G$ such that $\Ord(a)=\Ord(b)$, with respective conjugacy classes $C_a$ and $C_b$, we have for all $x\geq 2$: $$\pi(x;L/K;C_a)=\pi(x;L/K;C_b)\,.$$
        
        \item If $G=Q_8\times \Z/4\Z$, then denoting $a=(1,\Bar{2})$,$ b=(-1,0)$ and $C_a$, $C_b$ their respective conjugacy classes, the extension $L/K$ has an extreme Chebyshev bias relatively to $(C_a,C_b)$. More precisely, for all $x$ sufficiently large, we have: $$\pi(x;L/K;C_a)>\pi(x;L/K;C_b)\,.$$
    \end{enumerate} 
\end{cor}

As a second illustration, we present another family of groups giving rise to an extreme Chebyshev bias.


\begin{cor}\label{Example2}
     With notation as in Theorem  \ref{princ}, assume that $G=\left(\Z/p^m\Z \rtimes \Z/p^n\Z \right)\times H$, where $p$ is a prime, $1\leq n <m$ and $(H,+)$ is a finite group (\textit{e.g.} $G$ is a dihedral group). Let $a=(0,1,0)$, $b=(p^{m-n},0,0)$ and let $C_a$ and $C_b$ be their respective conjugacy classes. Then $L/K$ has an extreme Chebyshev bias relatively to $(C_a,C_b)$. More precisely, for all $x$ sufficiently large, we have: $$\frac{1}{|C_a|}\pi(x;L/K;C_a)>\frac{1}{|C_b|}\pi(x;L/K;C_b)\,.$$
\end{cor}


We end this paragraph by giving a group theoretic characterization of abelian groups for which the setting of the Cayley embedding produces an extreme Chebyshev bias (hence partially answering question {\bf b)} stated at the end of~\S\ref{subsec:cheb}).

Recall that a \textit{homocyclic $p$-group} is a direct product of cyclic $p$-groups of the same order (\emph{e.g.} elementary abelian $p$-groups, or cyclic $p$-groups).

\begin{thm}\label{Bias-Ab-gr}
     Let $G$ be a finite abelian group and let $k$ be a number field. Consider the injection $G \hookrightarrow \frak S (G)=:G^+$ (the group of permutations of $G$), given by the action of $G$ on itself by left translation. Let $L$ denote a Galois extension of $k$ with group $G^+$ and let $K=L^G$ be the subextension of $L/k$ fixed by $G$. Then there exists elements $a,b\in G$ with $\Ord(a)=\Ord(b)$ such that $L/K$ has an extreme Chebyshev bias relative to $(C_1=\{a\},C_2=\{b\})$ if and only if $G$ has a non-homocyclic $p$-Sylow subgroup for some prime $p$.
\end{thm}

This result has the following immediate consequence.
\begin{cor}
    Let $G$ be an abelian group such that for every prime $p$ the $p$-Sylow subgroup of $G$ is homocyclic (\textit{e.g.} $G$ is itself a homocyclic $p$-group). Then, case (1) of Theorem~\ref{princ} holds: for all $a,b\in G$ such that $\Ord(a)=\Ord(b)$ and for all $x\geq 2$, we have : $$\pi(x;L/K;\{a\})=\pi(x;L/K;\{b\})\,.$$
\end{cor}

\subsubsection{Extreme Chebyshev bias and low index}

We will obtain the following generalization of the construction of~\cite{FJ}*{\S2} leading to new instances of extreme Chebyshev bias. 

\begin{prop}\label{G+ab}
    Let $k$ be a number field and let $G=\Z/p^n\Z \times \Z/p^m\Z \times H$ where $p$ is a prime, $1\leq n < m$ and $(H,+)$ is a finite abelian group. Let $a=(0,1,0)$ and $b=(p^{m-n},0,0)$. Then, there exists a Galois extension $L/k$ of group $$G^+\simeq (\Z/p^m\Z \times \Z/p^m\Z \times H )\rtimes \Z/2\Z$$ (where the non-trivial element of $\Z/2\Z$ acts on the left factor by switching the coordinates of the first two factors of $\Z/p^m\Z \times \Z/p^m\Z \times H$) such that $L/L^G$ has an extreme Chebyshev bias relatively to $(\{a\},\{b\})$.
\end{prop}

We highlight the importance of the Artin formalism (notably the compatibility with induction) as a crucial tool for producing extreme Chebyshev biases. Proposition~\ref{gen} makes this feature precise and answers in turn in the affirmative question {\bf a)} stated at the end of \S\ref{subsec:cheb}. Proposition~\ref{G+ab} is mainly a consequence of Proposition~\ref{gen}. 
We will also prove (see Proposition~\ref{min-ab}) that in the notation of Proposition~\ref{G+ab}, the group $G^+$ minimizes the index $(G^+:G)$ of a group for which $t=\mathds{1}_{\{a\}}-\mathds{1}_{\{b\}}$ satisfies the conditions of Proposition~\ref{gen}.

\medskip
The next statement gives an explicit construction of a family of number fields $L/K/k$ such that $G^+:=\Gal(L/k)$ and $G:=\Gal(L/K)$ are of the same type as in Proposition~\ref{G+ab}. We obtain the following result the proof of which is postponed to the appendix. 
\begin{thm}\label{Galois}
    Let $m\geq 2$ and let $k$ be the maximal totally real subfield of $\Q \left (\zeta_{2^m} \right )$, where $\zeta_{2^m}$ is a primitive $2^m$-th root of $1$ in $\C$. 
    Let $q\geq 5$ be a prime congruent to $1$ modulo $4$ and let $a,b$ be positive integers such that $q=a^2+b^2$. Let $D_1$ and $D_2$ be the splitting fields inside $\C$ of $X^{2^m}-(a+ib)$ and $X^{2^m}-(a-ib)$ over $\Q(\zeta_{2^m})$, respectively. Then the compositum $L=D_1D_2$ contains $k$, it is Galois over $\Q$ and we denote $G^+:=\Gal(L/k)$ and $G:=H\Gal(D_1D_2/D_1)$ where $H$ is a non-trivial proper subgroup of $\Gal(D_1D_2/D_2)$. In this situation there exist two elements $\sigma_1$ and $\sigma_2$ in $G$ giving rise to an extreme Chebyshev bias in $L/L^G$. 
    Moreover, there exists an integer $n\in\{1,\ldots,m-1\}$ such that 
    \[
        G \simeq \Z/2^n\Z \times \Z/2^m\Z \text{ and } G^+ \simeq \left(\Z/2^m\Z \times \Z/2^m\Z \right) \rtimes \Z/2\Z\,.
    \]
\end{thm}





Finally we consider Galois extensions of number fields $L/k$ of prime power degree. In this situation we give a simple sufficient group theoretic condition on $\Gal(L/k)$ implying an extreme Chebyshev bias for some subextension $L/K$. 

\begin{thm}\label{p_groupsBias}
    Let $p$ be a prime number and let $G^+$ be a $p$-group. Let $L/k$ be a Galois extension of number fields with group $G^+$. Assume that there exists a $p$-th power $\sigma \in G^+$ such that the generated subgroup $\langle \sigma \rangle$ is not normal in $G^+$. Then, there exists an intermediate extension $K$ of $L/k$, two conjugacy classes $C_1,C_2$ of $\Gal(L/K)$, and a real number $A>0$, such that for all $x>A$ we have: $$\frac{1}{|C_1|}\pi(x;L/K;C_1)>\frac{1}{|C_2|}\pi(x;L/K;C_2)\,.$$
\end{thm}

Taking $p=2$ in Theorem~\ref{p_groupsBias}, we obtain the following family of examples.

\begin{cor}\label{p_groupsbiascor}
    Let $G^+=U(n,\Z/2^m\Z)$ be the unitriangular matrix subgroup of $\GL_n(\Z/2^m\Z)$, where $m\geq 2$ and $n\geq 3$, that is the $2$-subgroup of all upper triangular matrices with diagonal coefficients all equal to $1$. Let $L/k$ be a Galois extension of number fields of group $G^+$. Then, there exists a subextension $K$ of $L/k$ such that $L/K$ has an extreme Chebyshev bias.
\end{cor}

In the setting of Theorem~\ref{p_groupsBias}, we address the \emph{Linnik type question} that naturally comes to mind: can one give an explicit value of the real number $A$ in terms of arithmetic or Galois theoretic invariants of $L/K$? Using~\cite{FJ}*{Theorem 2.3}, we give an answer that is conditional on the Riemann Hypothesis. 


\begin{thm}\label{Linn1}
    With notation and assumptions as in Theorem~\ref{p_groupsBias}, set $k=\Q$ and assume the
    Riemann Hypothesis for the Dedekind Zeta function of $L$. Fix $C_1$ and $C_2$ as in the conclusion of Theorem~\ref{p_groupsBias}. For any fixed $b\in C_2$ let \( r \) be the number of conjugacy classes contained in \( \{x \in G \colon \exists g\in G,\, gx^pg^{-1}=b\} \). Denote $\mathrm{rd}_L:=d_L^{\frac{1}{[L:\Q]}}$, where $d_L$ is the Discriminant of $L$. For \( B > 0 \), a sufficiently large absolute constant, we set 
    \[ A_1 = B(\sqrt{r}\log(\mathrm{rd}_L+2)[L:\mathbb{Q}])^{2p}\,. \]
    Then for all \( x > A_1 \) we have 
    \[ \frac{1}{|C_1|}\pi(x;L/K;C_1)> \frac{1}{|C_2|}\pi(x;L/K;C_2) \, . \]
\end{thm}

\section{Cayley embedding: equal prime counting functions \emph{vs} extreme bias}\label{section:dicho}
The main goal of the section is to prove Theorem~\ref{princ}.

\subsection{Induction and consequences on prime ideal counting functions}
We will use the following notation. For $G$ a finite group and for $\ell \geq 1$, define $f_\ell\colon G\rightarrow G$ and $r_\ell\colon G\rightarrow \Z$ by 
\begin{equation}\label{eq:rl}
f_\ell(x)=x^\ell\,,\qquad r_\ell(x)=|\{g\in G\colon f_\ell(g)=g^\ell=x\}|\qquad (x\in G)\,.
\end{equation}
Recall also the definition of the usual scalar product on class functions on $G$:
\[
\langle f,g\rangle_G =\frac{1}{|G|}\underset{x\in G}{\sum}f(x) \overline{g(x)}\,.
\]
Unless this leads to confusion, we will simply write $\langle\cdot,\cdot\rangle$ instead of $\langle\cdot,\cdot\rangle_G$.

We first state an immediate consequence of Frobenius reciprocity. When $G$ is a subgroup of a group $G^+$ and $t$ is a class function of $G$ we denote $t^+:=\text{Ind}_G^{G^+}t$ the induced class function by $t$ on $G^+$. Recall that for all $a\in G^+$ we have: $$  t^+(a):=\frac{1}{|G|}\sum_{\substack{g\in G^+\\ g^{-1}ag\in G}}t(g^{-1}ag)$$
\begin{prop}\label{Ind-Rac}
Let $G$ be a subgroup of a finite group $G^+$ and let $t\colon G\to \C$ be a class function. Let $\ell \geq 1$ and assume that $(t\circ f_\ell)^+=0$. Then $\langle t,r_\ell \rangle =0$.
\end{prop}

\smallskip
\begin{proof}
We start by noticing that $\langle t , r_\ell \rangle = \langle t\circ f_\ell , 1\rangle $.
Indeed:
\[ \langle t,r_\ell\rangle=\frac{1}{|G|}\underset{g\in G}{\sum}t(g)r_\ell(g)=\frac{1}{|G|}\underset{g \in G^\ell}{\sum}\ \underset{\substack{a \in G,\\ a^\ell=g}}{\sum}t(g)=\frac{1}{|G|}\underset{a\in G}{\sum}t(a^\ell)=\langle t\circ f_\ell,1\rangle\,, \]
where $G^\ell$ denotes the set of $\ell$-th powers of $G$. We apply Frobenius reciprocity: $$\langle t\circ f_\ell,1\rangle_G=\langle (t\circ f_\ell)^+,1 \rangle_{G^+}\,,$$
which concludes the proof by our assumption.

\end{proof}

\begin{rk} \label{rem:converse}
\begin{enumerate}
    \item The converse of the previous result is false: we can even answer negatively to a weak converse:
    \begin{center}
    \emph{
    If $t^+=0$ and if for all $k|\ell$ we have $\langle t,r_k \rangle=0$, can we deduce $(t\circ f_\ell)^+=0$ ?} 
    \end{center}
    Consider indeed $Q_8$ the quaternion group of order $8$. Fix two distinct elements $i,j \in Q_8$ of order $4$ generating $Q_8$. The diagonal injection $\langle i \rangle \times \langle j \rangle \hookrightarrow Q_8 \times Q_8$ gives rise to an injection $G:=\langle i\rangle \times \langle j \rangle \hookrightarrow (Q_8)^2 \rtimes \Z/2\Z=G^+$, where the semi-direct product structure corresponds to the permutation of coordinates in $(Q_8)^2$. Consider $t=\mathds 1_{(-1,1)}-\mathds 1_{(1,-1)}$, we have $r_2(-1,1)=r_2(1,-1)$ and $t^+=0$ but $(t\circ f_2)^+\ne 0$.
    \item We will see in Lemma \ref{récip} that when $G^+$ is well chosen (in particular if $G\hookrightarrow\mathfrak S(G)\simeq G^+$ is the Cayley embedding, as in Theorem~\ref{princ}), the weak converse above becomes true. This is a crucial point for the proof of Theorem \ref{princ}.
\end{enumerate}
\end{rk}

Building on Proposition~\ref{Ind-Rac}, Proposition~\ref{gen} below draws consequences regarding prime ideal counting in number fields depending on whether
\begin{enumerate}
    \item[(a)] one has $(t\circ f_\ell)^+=0$ for any $\ell$ square-free and less than a square-free integer $d$ for which $\langle t,r_d\rangle\neq 0$,
    \item[(b)] one has $(t\circ f_\ell)^+=0$ for all square-free $\ell$.
\end{enumerate}


\begin{prop}\label{gen} Let $L/K/k$ be a tower of number field extensions and assume that $L/k$ is Galois. Let $G=\Gal(L/K)$ and $G^+=\Gal(L/k)$. Let $t\colon G\to \C$ be a class function. We have the following.
 \begin{enumerate}
     \item Assume there exists $d \geq 2$ square-free such that $\langle t, r_d \rangle \ne 0$ and that for $1\leq \ell < d$ square-free, one has $(t \circ f_\ell)^+=0$. Then we have: $$\pi(x;L/K;t)=\mu(d)\langle t, r_d \rangle \frac{x^{\frac{1}{d}}}{\log x }+o\left(\frac{x^\frac{1}{d}}{\log x}\right)\,$$
    where $\mu$ is the Möbius function.
    \item Assume that for all square-free $\ell\geq1$ we have $(t\circ f_\ell)^+=0$. Then, for every $x\geq 2$, we have:
    $$\pi(x;L/K;t)=0\,.$$
 \end{enumerate}
    
\end{prop}

Part (1) of the statement is a generalization of \cite{FJ2}*{Lemma 3.1}. Note also that Proposition~\ref{gen} already exhibits the type of dichotomy stated in Theorem~\ref{princ}.

For the proof we introduce the analogue of Chebyshev's $\theta$ and $\psi$ functions in the context of Galois extensions of number fields:
\begin{equation}\label{eq:phipsi}
\theta(x;L/K;t)=\sum_{\substack{\frak{p}\lhd O_K,\\ N\frak{p}\leq x}}t(\varphi_\frak{p})\log(N\frak{p})\,, \qquad
\psi(x;L/K;t):=\sum_{\substack{\frak{p}\lhd O_K,\, k\geq 1 \\ N\frak{p}^k\leq x}}t(\varphi_\frak{p}^k)\log(N\frak{p})\,.
\end{equation}

\begin{proof}[Proof of Proposition~\ref{gen}]
We follow closely the proof of~\cite{FJ2}*{Lemma 3.1}.

By the inclusion-exclusion principle, we have for $x\geq2$:
\begin{equation}\label{inc-excl}
    \theta(x;L/K;t)=\underset{\ell\geq 1}{\sum}\mu(\ell)\psi(x^{\frac{1}{\ell}};L/K; t \circ f_\ell)\,.
\end{equation}
By invariance of Artin $L$-functions under induction we have for $x\geq2$:
\begin{equation}\label{Induction}
    \forall \ell\geq1\,,\qquad  \psi(x^{\frac{1}{\ell}};L/K;t \circ f_\ell)=\psi(x^{\frac{1}{\ell}},L/k, (t \circ f_\ell)^+ )\,.
\end{equation}
\begin{enumerate}
    \item  
By our assumption and by~\eqref{Induction}
for all $1\leq \ell < d$ square-free: 
$$\psi(x^{\frac{1}{\ell}};L/K;t \circ f_\ell)=0\,.$$ 
By Chebotarev's Theorem we have: 
$$\psi(x^{\frac{1}{d}};L/K; t \circ f_d)-\langle t,r_d \rangle x^{\frac{1}{d}}  =o( x^{\frac{1}{d}})\,.$$
Moreover one has the upper bound:
$$\left|\sum_{k=d+1}^{\log(x)} \mu(\ell) \psi(x^{\frac{1}{\ell}};L/K;t\circ f_\ell)\right| \ll x^{\frac{1}{d+1}}+\log(x)x^{\frac{1}{d+2}}\ll x^{\frac{1}{d+1}}\,.$$
We deduce that
\[
\theta(x;L/K;t)=\mu(d)\langle t,r_d \rangle x^{\frac{1}{d}}+o( x^{\frac{1}{d}})\,.
\]
Summing by parts as in the proof of~\cite{FJ2}*{Lemma 3.1} gives the result.
    \item Again by assumption combined with~\eqref{inc-excl} and~\eqref{Induction}, we have for $x\geq2$: $$\theta(x;L/K;t)=\underset{\ell\geq 1}{\sum}\mu(\ell)\psi(x^{\frac{1}{\ell}};L/K;t\circ f_\ell)=0\,.$$ 
A summation by parts allows us to conclude. 
\end{enumerate}


\end{proof}

  If $C_1$ and $C_2$ are two conjugacy classes of $G$, consider $t_{C_1,C_2}$ (see~\eqref{def:C1C2}). We have $t_{C_1,C_2}^+=0$ if and only if $C_1$ and $C_2$ are contained in the same conjugacy class of $G^+$; indeed if $C\subset G$ is a conjugacy class contained in a conjugacy class $C^+$ of $G^+$, then
\begin{equation}\label{eq:indconj}
\left(\mathds 1_C\right)^+=\frac{|C||G^+|}{|G||C^+|}\mathds 1_{C^+}\,.
\end{equation}

 Note the following.
  \begin{enumerate}
      \item If $t_{C_1,C_2}$ satisfies the hypotheses of Proposition \ref{gen}(1), then $L/K$ exhibits an extreme bias relative to $C_1$ and $C_2$. Note also that when $d$ is the product of an odd number of prime numbers (for example, if $d=p$ is prime), then the Frobenius elements will be preponderant in the conjugacy class containing the fewest $d$-th roots of elements in $G$. However, when $d$ is a product of an even number of prime numbers, the bias is towards the class with more $d$-th roots.
      \item If $t_{C_1,C_2}$ satisfies the hypotheses of Proposition \ref{gen}(2), then for all $x\geq 2$ we have : $$ \frac{|G|}{|C_1|}\pi(x;L/K;C_1)=\frac{|G|}{|C_2|}\pi(x;L/K;C_2)\,.$$ 
  \end{enumerate}
  
The conditions $(t\circ f_\ell)^+=0$ for square free $1\leq \ell < d$ and $\langle t,r_d\rangle \ne 0$ in Proposition~\ref{gen} are incompatible if $G$ is normal in $G^+$. Indeed, if $G\lhd G^+$ and $t^+=(t\circ f_1)^+=0$, then for all $k\geq 1$ we have $\langle t,r_k \rangle =0$ (see~\cite{FJ}*{Corollary 3.10}).  Corollary \ref{Normal} discusses the consequence on prime ideal counting functions. 

\begin{cor}\label{Normal}
    With notation as in Proposition~\ref{gen}, assume $G\lhd G^+$ and 
    $t^+=0$. For all $x\geq2$, we have:
    $$\pi(x;L/K;t)=\pi_\nr(x;L/K;t)=0\,,$$
    where $\pi_\nr(x;L/K;t)$ is defined as $\pi(x;L/K;t)$ to which one subtracts the contribution of the prime ideals of $O_K$ ramifying in $L$.
\end{cor}

\begin{proof}
We first prove that  
$(t\circ f_\ell)^+=0$ for all $\ell\geq 1$.
 Let $\ell\geq1$ and $g \in G^+$. If $g\notin G$ then $aga^{-1}\notin G$ for all $a\in G^+$, hence $(t\circ f_\ell)^+(g)=0$. Otherwise $g\in G$, thus $aga^{-1} \in G$ for all $a \in G$. Therefore $(t\circ f_\ell)^+(g)=t^+\circ f_\ell (g)=0$. 
 
 By (2) of Proposition~\ref{gen} we deduce that for all $x\geq2$: $$\pi(x,L/K,t)=0\,.$$
In what follows we use the letter $\frak{p}$ for primes of $O_K$ and $\frak{P}$ for primes of $O_L$.\\
Denote $$\mathbb{P}:=\{\frak{a}\lhd O_k\ :\ \forall \frak{p}|\frak{a},\ \frak{p}\text{ is unramified in }L\}\,,$$
where $\frak{p|a}$ means that the prime $\frak{p}$ is above $\frak{a}$.\\
By assumption $K/k$, is Galois, thus for any non-zero $\frak{a}\lhd O_k$ we denote by $f_{\frak{a}}$ the common residual degree associated to any prime ideal $\frak p$ of $O_K$ above $\frak a$ (i.e. $N\frak{p}=(N\frak{a})^{f_\frak{a}}$).

We have: $$\pi_{\nr}(x;L/K;t)=\sum_{{\substack{\frak{a}\in \mathbb{P}\\(N\frak{a})^{f_\frak{a}}\leq x}}}\sum_{\frak{p}|\frak{a}}t(\varphi_\frak{p})\,. $$
It is sufficient to prove that for any $\frak a\in\mathbb P$, one has $$\underset{\frak{p}|\frak{a}}{\sum}t(\varphi_\frak{p})=0\,.$$
Fix $\frak{a}\in \mathbb{P}$, $\frak{p}_0\lhd O_K$ and $\frak{P}_0\lhd O_L$ such that $\frak{p}_0|\frak{a}$ and $\frak{P}_0|\frak{p}$. 
If $\frak{p}\lhd O_K$ and $\frak{P}\lhd O_L$ are such that $\frak{p}|\frak{a}$ and $\frak{P}|\frak{p}$, denote $G_{\frak{P/a}} \subset G^+$, $G_{\frak{P/p}}\subset G$ the corresponding decomposition groups and $\varphi_{\frak{P/a}}$, $\varphi_{\frak{P/p}}$ the corresponding Frobenius elements. We have $\varphi_{\frak{P/p}}=\varphi_{\frak{P/a}}^f$, where $f:=f_\frak{a}$.\\

We have \begin{align*}
    t^+(\varphi_{\frak{P}_0/\frak{p}_0})&=\frac{1}{|G|}\sum_{g\in G^+}t\left(\left(g\varphi_{\frak{P}_0/\frak{a}}g^{-1}\right)^f\right)=\frac{1}{|G|}\sum_{\frak{P}|\frak{a}}|G_{\frak{P}_0/\frak{a}}|t\left(\varphi_{\frak{P/a}}^f\right)\\
    &=\frac{1}{|G|}\sum_{\frak{p}|\frak{a}}|G_{\frak{P}_0/\frak{a}}|\sum_{\frak{P|p}}t\left(\varphi_{\frak{P/p}}\right)=\frac{|G_{\frak{\frak{P}_0}/\frak{a}}|}{|G_{\frak{P}_0/\frak{p}_0}|}\sum_{\frak{p}|\frak{a}}t(\varphi_\frak{p})\,.
\end{align*}
The result follows since $t^+=0$. 

\end{proof}

\begin{rk}  Note that when we have $\pi(x;L/K;t)=0$ for all $x\geq 2$, a simple induction allows us to show that for any prime $p\geq 2$, we have $\underset{\frak{p}|p}{\sum}t(\varphi_\frak{p})=0$. Thus, if $L/K/\Q$ is a tower of Galois extensions (which means the situation of Corollary~\ref{Normal} with $k=\Q$) we can conclude directly that $\pi_\nr(x;L/K;t)=0$ from $\pi(x;L/K;t)=0$.
\end{rk}

\subsection{Proof of Theorem \ref{princ}}
We will see that Theorem \ref{princ} is a consequence of Theorem \ref{OC-thm} which uncovers a group theoretic property implying the dichotomy given in Theorem \ref{princ}. Making precise which group theoretic properties are involved requires some preparation.

\begin{lem}\label{1}
    Let $G$ be a finite group, $a \in G$, and $\ell \geq 2$. Set $n = \Ord(a)$, $d = \gcd(\ell, n)$, and $\ell = d\ell'$. If $x$ is an $\ell$-th root of $a$, then $\Ord(x)=kdn$ for some divisor $k$ of $\ell'$.
\end{lem}

\begin{proof}
We have $\Ord(x) = \gcd(\Ord(x), \ell) \Ord(a)$. Now, $\gcd(\Ord(x), \ell)$ is a multiple of $d$ that divides $\ell$, so $\gcd(\Ord(x), \ell) = kd$ where $k | \ell'$. 

\end{proof}

For an element $a$ in a finite group $G$ and for $\ell, s \geq 1$, we set
\[
r_{\ell,s}(a)=|\{x\in G\colon x^\ell=a,\, \Ord(x)=s\}|\,.
\]
This defines a class function on $G$.


\begin{lem}\label{nb-rac}
    Let $G$ be a finite group, and let $a$ and $b$ be elements of $G$ of the same order. Let $\ell \geq 1$ be a square-free integer and suppose that for every divisor $k$ of $\ell$, we have $r_k(a) = r_k(b)$. Then, for every $s \geq 1$, 
    $$r_{\ell, s}(a) = r_{\ell, s}(b)\,.$$
\end{lem}

\begin{proof}
We proceed by induction on $\ell$. The case $\ell=1$ is trivial; fix $\ell \geq2$.

Let $n = \Ord(a)$, $d = \gcd(n, \ell)$, and $\ell = d\ell'$. By Lemma~\ref{1} any $\ell$-th root $x$ of $a$ in $G$ has order $\Ord(x)=kdn$ for some divisor $k$ of $\ell'$. Let $k | \ell'$ with $k < \ell'$ since $\ell$ is square-free, the integers $\frac{\ell'}{k}$ and $kdn$ are coprime. Hence there exists $u\in \Z$ such that $$u\frac{\ell'}{k}\equiv 1 \mod kdn\,. $$ For $m\geq1$ denote $R_{m}:=\{y\in G\ :\ y^m=a,\ \Ord(y)=kdn\}$. Define $$\iota:R_{kd}\to R_{\ell}\,,\quad x \mapsto x^{u}\,.$$ $\iota$ is well defined since $u$ is coprime to $kdn$ and $$\left(x^{u}\right)^\ell=(x^{u\frac{\ell'}{k}})^{kd}=x^{kd}=a\,.$$
Moreover, $\iota$ is injective since: if $x^u=z^u$ where $\Ord(x)=\Ord(z)=kdn$, then $x=x^{u\frac{\ell'}{k}}=z^{u\frac{\ell'}{k}}=z$.\\ Finally $\iota$ is surjective, since: if $y\in R_{\ell}$, then $x=y^{\frac{\ell'}{k}}\in R_{kd}$ (because $x^{kd}=y^\ell=a$ and $\gcd (\frac{\ell'}{k},kdn)=1$ ). Hence: $$\iota(x)=y^{u\frac{\ell'}{k}}=y\,.$$
Thus, $r_{\ell, kdn}(a) = r_{kd, kdn}(a)$, and similarly, $r_{\ell, kdn}(b) = r_{kd, kdn}(b)$.\\
The induction hypothesis, implies $$r_{\ell, kdn}(a)=r_{kd, kdn}(a) = r_{kd, kdn}(b)=r_{\ell, kdn}(b)\, .$$
Now, consider the case where $k = \ell'$. We have $$r_{\ell, l'dn}(a) = r_\ell(a) - \sum_{\substack{k | \ell'\\ k < \ell'}} r_{\ell, kdn}(a) = r_\ell(b) - \sum_{\substack{k | \ell'\\ k < \ell'}} r_{\ell, kdn}(b) = r_{\ell, \ell'dn}(b)\, .$$ The induction is now established. 
\end{proof}

In what follows we will use the following notation: Given a conjugacy class $C$ of a group $G$, we denote by $\Ord(C)$ (and we will call \emph{order of $C$}) the order of any (every) element of $C$.\\
Now, we can state and prove the following ``weak converse'' of Proposition \ref{Ind-Rac} (recall Remark~\ref{rem:converse}).

\begin{lem}\label{récip}
    Let $G^+$ be a finite group and $G$ a subgroup of $G^+$ such that 
    
     \begin{center} $(\star)$ \emph{Two elements $a,b\in G$ have the same order if and only if they are conjugate in $G^+$.}
    \end{center}
    Let $\ell \geq 2$ be a square-free integer and let $C_1,C_2$ be conjugacy classes of $G$ such that $\Ord(C_1)=\Ord(C_2)$. Consider the class function $t_{C_1,C_2}$ (see~\eqref{def:C1C2}) on $G$. If $\langle t_{C_1,C_2}, r_k \rangle = 0$ for all divisors $k$ of $\ell$, then $$(t_{C_1,C_2} \circ f_\ell)^+ = 0\,.$$
\end{lem}

\begin{proof}
 We will write $r_{\ell,s}(C)$ and $r_\ell(C)$ for the common value respectively assumed by $r_{\ell,s}$ and $r_\ell$ on the conjugacy class $C$ of $G$.

Suppose $r_\ell(C_1) = r_\ell(C_2) > 0$. Denote $n = \Ord(C_1)=\Ord(C_2)$, $d = \gcd(\ell, n)$, and $\ell = d\ell'$. For $i \in \{1, 2\}$ and $k | \ell'$, let $$D_{i, k} := \{x \in G \colon x^\ell \in C_i,\, \Ord(x) = kdn\}\,.$$
We have, for $i \in \{1, 2\}$ and $k | \ell'$, 
$$|D_{i, k}| = |C_i| r_{\ell, kdn}(C_i)\,.$$ 
According to Lemma \ref{nb-rac}, this implies that for all $k | \ell'$, 
\begin{equation}\label{Classeq}
    \frac{|D_{1, k}|}{|C_1|} = \frac{|D_{2, k}|}{|C_2|}\,.
\end{equation}
Set $\varphi_k = \frac{|G|}{|C_1|} \mathds 1_{D_{1, k}} - \frac{|G|}{|C_2|} \mathds 1_{D_{2, k}}$. For simplicity, write $t=t_{C_1,C_2}$. We express 
$$
t \circ f_\ell = \frac{|G|}{|C_1|} \sum_{k | \ell'} \mathds 1_{D_{1, k}} - \frac{|G|}{|C_2|} \sum_{k | \ell'} \mathds 1_{D_{2, k}} = \sum_{k | \ell'} \varphi_k\,.
$$
To show that $(t \circ f_\ell)^+ = 0$, it suffices to demonstrate that for all $k | \ell'$, we have $\varphi_k^+ = 0$.

Fix a divisor $k$ of $\ell'$. The sets $D_{i, k}$, for $i \in \{1, 2\}$, are conjugacy invariant in $G$. Thanks to the equality \eqref{Classeq}, either both $D_{1, k}$ and $D_{2,k}$ are empty, in which case $\varphi_k=0$, and hence $\varphi_k^+ = 0$, or both $D_{1, k}$ and $D_{2,k}$ are non-empty, and in this case we can write $D_{i, k} = \bigcup_{1 \leq j \leq r_i} F_{j, i}$ where the $F_{j, i}$ are conjugacy classes in $G$. By construction, for all $i, j$, we have $\Ord(F_{j, i}) = kdn$. Thus, there exists a conjugacy class $F^+$ in $G^+$ that contains all $F_{i, j}$. 
Recalling~\eqref{eq:indconj}, we have: 
\begin{align*}
    \varphi_k^+ &= \frac{|G|}{|C_1|} \sum_{1 \leq j \leq r_1} \frac{|F_{j, 1}||G^+|}{|G||F^+|} \mathds 1_{F^+} - \frac{|G|}{|C_2|} \sum_{1 \leq j \leq r_2} \frac{|F_{j, 2}||G^+|}{|G||F^+|} \mathds 1_{F^+}\\ &= \frac{|G^+|}{|F^+|} \left(\frac{|D_{1, k}|}{|C_1|} - \frac{|D_{2, k}|}{|C_2|}\right) = 0 \, ,
\end{align*}
which proves the lemma. 
\end{proof}

We are now ready to prove a generalization of Theorem~\ref{princ} where the relevant group theoretic property we use is highlighted.

\begin{thm}\label{OC-thm}
    Let $L/K/k$ be a tower of number fields such that $L/k$ is Galois. Denote $G=\Gal(L/K)$ and $G^+=\Gal(L/k)$. Assume that the couple $(G,G^+)$ satisfies $(\star)$ (see Lemma~\ref{récip}).
    
    Let $C_1$ and $C_2$ be conjugacy classes of $G$ such that $\Ord(C_1)=\Ord(C_2)$. Then two cases occur:
    \begin{enumerate}
        \item either for every square-free $\ell$, one has $r_\ell(C_1)=r_\ell(C_2)$, in which case, for all $x\geq 2$ one has
        $$\frac{|G|}{|C_1|}\pi(x;L/K;C_1)=\frac{|G|}{|C_2|}\pi(x;L/K;C_2)\,,$$
        \item or there exists $d\geq2$ square-free such that $r_d(C_1)\ne r_d(C_2)$, in which case there exists $A>0$ such that, up to exchanging $C_1$ and $C_2$, we have for all $x\geq A$,
        \[ \frac{|G|}{|C_1|}\pi(x;L/K;C_1)>\frac{|G|}{|C_2|}\pi(x;L/K;C_2)\,.
        \]
        Hence $L/K$ has in this case an extreme Chebyshev bias relative to $(C_1,C_2)$.
    \end{enumerate}
\end{thm}
\begin{proof}
    Let $C_1$ and $C_2$ be two conjugacy classes in $G$ such that $\Ord(C_1)=\Ord(C_2)$. Let $t=t_{C_1,C_2}$ (recall~\eqref{def:C1C2}) so that $t^+=0$ by~\eqref{eq:indconj}.  In case (1), we use Lemma~\ref{récip} to conclude that for all square-free $\ell\geq2$, we have $(t\circ f_\ell)^+=0$. Therefore, by Proposition~\ref{gen}(2) we have $\pi(x;L/K;t)=0$ for all $x\geq2$.

    In case (2), we use again Lemma~\ref{récip} to deduce that for all square-free $\ell<d$ we have $(t\circ f_\ell)^+=0$. As $\langle t,r_d\rangle=r_d(C_1)-r_d(C_2) \ne 0$, we can apply Proposition~\ref{gen}(1) to conclude.
\end{proof}

We can now deduce Theorem~\ref{princ} from Theorem~\ref{OC-thm}.

\begin{proof}[Proof of Theorem~\ref{princ}]
Let $G$ and $G^+$ be as in the statement of Theorem \ref{princ}. By Hilbert's work, the inverse Galois problem over $k$ is solved for $G^+$, and therefore we can consider $L/k$ a Galois extension of group isomorphic to $G^+$. Denoting $K=L^G$, it suffices to prove that the tower $L/K/k$ satisfies the hypotheses of Theorem \ref{OC-thm}. Thus, it suffices to prove that any two elements of $G$ having the same order are conjugates in $G^+$. This is an easy consequence of Cayley's Theorem, for completeness we reprove it here. Let $\Psi \colon G \rightarrow \frak{S}(G)$ be the injective morphism induced by the action of $G$ on itself by left translation. Let $a,b \in G$ be elements of the same order $r$.
The orbits of the permutation $\Psi(a)$ are exactly the orbits of the action of $\langle a \rangle$ on $G$ by left translation. These orbits have the form $\{a^ig\colon 1\leq i\leq r\}$ with $g \in G$. They all have the same cardinality, which is $r$, so $\Psi(a)$ is a product of cycles of order $r$, and the same holds for $\Psi(b)$. Therefore, the permutations $\Psi(a)$ and $\Psi(b)$ have the same cycle type in $\mathfrak S(G)$, so they are conjugate. Conversely, two conjugate elements must have the same order, and an injective morphism preserves the order.

\end{proof}

\begin{rk} One might wonder, given a group $G$, which groups $G^+$ satisfy condition $(\star)$ in Lemma~\ref{récip}. Obviously, any group containing $\frak{S}(G)$ satisfies this property. Since we are looking for ``minimal'' such groups, we will restrict ourselves, for simplicity, to the case of an abelian $p$-group of the form $G=\Z/p^m\Z\times \Z/p^n\Z$ (however the idea works more generally for all abelian groups and also for some non-abelian groups) and describe a way of constructing a suitable $G^+$.
\begin{enumerate}
    \item First one has the natural injection $G \hookrightarrow \left(\Z/p^m\Z\right)^2$.
    \item Next notice that $\left(\Z/p^m\Z\right)^2$ (and any homocyclic $p$-group) has the property that for any elements $a,b \in \left(\Z/p^m\Z\right)^2 $ of the same order, there exists an automorphism $f \in \mathrm{Aut} \left((\Z/p^m\Z)^2\right)$ such that $f(a)=b$.
    \item Consider $G^+=\left(\Z/p^m\Z\right)^2 \rtimes \text{Aut} \left((\Z/p^m\Z)^2\right)$: here the semi-direct product is given by the natural action of $\text{Aut} \left((\Z/p^m\Z)^2\right)$ on $\left(\Z/p^m\Z\right)^2$.
    \item It is easy to verify that if $a,b\in G$ have the same order and if $f\in \mathrm{Aut} \left((\Z/p^m\Z)^2\right)$ is such that $f(a)=b$, then one has $(1,f)(a,1)(1,f)^{-1}=(b,1)$.
\end{enumerate}
The drawback of this construction is the inverse Galois problem it leads to: is $\left(\Z/p^m\Z\right)^2 \rtimes\mathrm{Aut} \left((\Z/p^m\Z)^2\right)$ the Galois group of an extension of $\Q$ (or any number field $k$)?
\end{rk}

We turn to the proof of Corollary~\ref{Example1}.

\begin{proof}[Proof of Corollary~\ref{Example1}]
    \begin{enumerate}
        \item Assume $G=Q_8$ the quaternion group of order $8$. If $C_1\ne C_2$ are two conjugacy classes of the same order, then they both have order $4$ and hence for $\ell$ square-free, if $\ell$ is odd then $r_\ell(C_1)=1=r_\ell(C_2)$, and if $\ell$ is even then $r_\ell(C_1)=0=r_\ell(C_2)$. This is exactly the first case of Theorem \ref{OC-thm} which yields the result.
        \item Assume $G=Q_8\times H$ where $H=\langle a \rangle$ is a cyclic group of order $4$. Let $a=(-1,1)$ and $b=(1,a^2)$. We have $r_2(a)=12>4=r_2(b)$, so the result is a consequence of Theorem \ref{OC-thm}(2). 
    \end{enumerate}
\end{proof}

\section{A group theoretic property implying extreme Chebyshev biases}\label{Property_P}
In this section we will study the property $\mathcal{P}(d)$ below. Its relevance in producing extreme Chebyshev biases comes from Proposition~\ref{gen}(1) and this property has already been used to prove Theorem~\ref{OC-thm}. More precisely, when used in conjonction with $(\star)$ (see Lemma~\ref{récip}), a group satisfying $\mathcal P(d)$ gives rise to an extreme Chebyshev bias (provided one can solve the inverse Galois problem relevant to the situation).


\begin{defi}[Property $\mathcal{P}(d)$]\label{P} 
Let $G$ be a finite group and $d\geq 2$ a square-free integer. We say that $G$ satisfies $\mathcal P(d)$ if there exist $a,b\in G$ of the same order such that $0=r_d(a)< r_d(b)$ (recall~\eqref{eq:rl}) and that for all square-free $1< \ell <d$ one has $r_\ell(a)=r_\ell(b)$.

For simplicity, we will write ``$G$ satisfies $\mathcal P$'' when $G$ satisfies $\mathcal P(2)$.
\end{defi}


The following statement makes precise the way in which one can use $\mathcal P(d)$ to obtain extreme Chebyshev biases.

\begin{prop}\label{prop:PimpliesBias}
Let $k$ be a number field and let $G$ be a finite group satisfying $\mathcal P(d)$ for some $d\geq 2$. Then there exists a Galois extension $L/k$ and an embedding $G\hookrightarrow \Gal(L/k)$ such that $L/L^G$ has an extreme Chebyshev bias (in other words, case (2) of Theorem~\ref{princ} occurs for some conjugacy classes $C_1$ and $C_2$ of $G$).
\end{prop}

\begin{proof}
Let $G^+$ be a group containing $G$ such that $(G,G^+)$ satisfies condition $(\star)$ of Lemma~\ref{récip} and such that the inverse Galois problem for $G^+$ is solved over $k$ (\emph{e.g.} consider the Cayley embedding $G\hookrightarrow G^+=\mathfrak S(G)$ as explained in the proof of Theorem~\ref{princ}). Then, since $G$ satisfies $\mathcal P(d)$, Theorem~\ref{OC-thm}(2) applies and the proof is complete.
\end{proof}

    

Note that Property $\mathcal P(d)$ is not necessary for a group to produce an extreme Chebyshev bias: for example, Corollary~\ref{Example1}(2) shows that there exist extensions with an extreme Chebyshev bias relatively to conjugacy classes of some elements $a$ and $b$ where we have $0<r_2(a)<r_2(b)$.

The following Proposition gives an idea about the minimal size of groups satisfying $\mathcal{P}(p)$ when $p$ is a prime
\begin{prop}\label{prop:p3}
    Let \(p\) be a prime, and let \(G\) be a group satisfying the property \(\mathcal{P}(p)\) respectively. Then, \(|G| \geq p^3\).

\end{prop}
\begin{proof}
Assume that $G$ satisfies $\mathcal{P}(p)$ relatively to $(a,b)$. Since $a$ has no $p$-th roots, then $\Ord(b)=\Ord(a)$ is a multiple of $p$. Let $c$ be a $p$-th root of $b$, by lemma~\ref{1}, $\Ord(c)=p\Ord(b)$. We may assume that $G=\langle a, c \rangle$ and that $\Ord(a)=pd$ with $d$ coprime to $p$. Denote $\alpha:=c^d$, we have $\Ord(\alpha)=p^2$. By Theorem~\ref{Ab} $G$ cannot be cyclic, thus $|G|\geq 2p^2$. This proves the result when $p=2$. 

Assume that $p\geq 3$, and let $S$ be a $p$-Sylow subgroup of $G$ containing $\alpha$. If $|S|>p^2$ then $p^3$ divides $|G|$. Otherwise, we prove that $S=\langle \alpha \rangle$ is not normal in $G$.\\
By contradiction, assume that $\langle \alpha \rangle \lhd G$. We have $\langle a^d\rangle =\langle b^d \rangle = \langle \alpha ^p \rangle$. Thus, $\alpha^p \in \mathcal{Z}(G)$ is central. Moreover, the interior automorphism of $G$ given by $a^p$ induces an automorphism $\phi$ of $\Z/p^2\Z$ that fixes $p\Z/p^2\Z$ (since $\langle \alpha \rangle \lhd G$ has order $p^2$). There exists $0\leq u \leq p^2-1$ coprime to $p$ such that, for all $x\in \Z/p^2\Z$ we have $\phi(x)=ux$. As $\phi(p)=p$, then $p$ divides $u-1$. Thus, $u \in \{kp+1\ :\ 0\leq k \leq p-1 \}$. If $u=kp+1$, with $k\geq 1$, then, $\Ord(\phi)=p$ since $$(kp+1)^p \equiv 1 (\text{ mod }p^2)\, ,$$
this is not possible, since $\Ord(\phi)$ divides $\Ord(a^p)$ which is coprime to $p$. Thus, $u=1$. Hence, $a^p$ commutes with $\alpha$. By Bezout, there exists $u,v \in \Z$ such that $up+vd=1$. We have \begin{equation}\label{decomp}
    a=a^{dv}a^{pu}\, .
\end{equation}
Thus, $a$ commutes with $\alpha$. Since $a^d \in \langle \alpha ^p \rangle $ has a $p$-th root in $\langle \alpha \rangle$, then, by \eqref{decomp}, $a$ has a $p$-th root, which contradicts our assumption. Thus, $\langle \alpha \rangle$ is not normal in $G$. The number of $p$-Sylows of $G$ is, thus, atleast $p+1$ (since it is $\equiv 1 \text{ mod }p) $), since each $p$-Sylow contains $p(p-1)$ generators, thus $|G| \geq p(p-1)(p+1)=p^3-p$. Hence, as $|G|$ is a multiple of $p^2$, we have $|G|\geq p^3$.

\end{proof}

The remainder of this Section is split into an abstract study of $\mathcal P(d)$ (\S\ref{sec:studyP}) and a presentation of several classes of groups satisfying Property $\mathcal P(d)$ (\S\ref{sec:exchar}) and therefore producing examples of extreme Chebyshev biases by Proposition~\ref{prop:PimpliesBias}.


\subsection{Preliminary group theoretic results}\label{sec:studyP}
 We start by simplifying the definition of Property $\mathcal{P}(p)$ in the cases we will be mostly interested in (that is, in the case $p$ is prime).
 
\begin{prop}\label{Psimp}
    Let $G$ be a finite group and let $a,b \in G$ have the same order. If $p=2$ or if $G$ is a $p$-group, then $G$ satisfies $\mathcal{P}(p)$ relatively to $(a,b)$ if and only if:
    \begin{enumerate}
        \item $a$ has no $p$-th roots in $G$,
        \item $b$ has $p$-th roots in $G$.
    \end{enumerate}
\end{prop}

\begin{proof}
The direct implication is obvious. For the converse, we only need to prove that for square-free $2\leq \ell< p$ we have $r_\ell(a)=r_\ell(b)$.

If $p=2$ there is nothing to prove. Otherwise $p>2$ and in this case $G$ is a $p$-group. For every $2\leq \ell < p$ square-free, $\ell$ is coprime with the order of the group $G$ and therefore $f_\ell$ is a permutation of $G$. We deduce that $r_\ell(g)=1$ for all $g\in G$, and in turn that the condition $r_\ell(a)=r_\ell(b)$ always holds. We conclude that $G$ satisfies $\mathcal{P}(p)$.

\end{proof}

The following lemma shows that Property $\mathcal{P}(d)$ still holds when performing the direct product with any group (mostly this will be used in the case $d=p$ is prime).
\begin{lem}\label{Direct-product}
    Let $G_1, G_2$ be finite groups and let $d\in\N$ be a square-free integer. If $G_1$ satisfies $\mathcal{P}(d)$ relatively to $(a,b)$, then $G_1\times G_2$ satisfies $\mathcal{P}(d)$ relatively to $((a,1),(b,1))$.
\end{lem}
\begin{proof}
    For $\ell\geq 2$ and $x\in G_1$ we have $$r_\ell(x,1)=\eta_\ell \cdot r_\ell(x)\, ,$$ where $\eta_\ell\geq1$ is the number of $\ell$-th roots of $1$ in $G_2$. Thus, $0=r_d(a,1)=r_d(a)<\eta_d \cdot r_d(b)=r_d(b,1)$ and for square-free $2\leq \ell <d$ we have $r_\ell(a,1)=r_\ell(b,1)$. Hence, $G_1\times G_2$ satisfies $\mathcal{P}(d)$ relatively to $((a,1),(b,1))$.
\end{proof}

Thanks to this lemma, we can now prove Corollary~\ref{Example2}.

\begin{proof}[Proof of Corollary~\ref{Example2}]
Let $1\leq n<m$, by Lemma~\ref{Direct-product} it is sufficient to prove that $G=\Z/p^m\Z \rtimes \Z/p^n\Z$ satisfies Property $\mathcal{P}(p)$ relatively to $a=(0,1)$, $b=(p^{m-n},0)$. Since the element $a$ is not a $p$-th power in $G$, while the element $b$ is a $p$-th power, we conclude by Proposition~\ref{Psimp} combined with Proposition~\ref{prop:PimpliesBias}.

\end{proof}

Next we state and prove (for completeness, in view of their importance) two elementary lemmas. The first one is a direct consequence of Bézout's theorem. 

\begin{lem}\label{2}
Let $G$ be a finite abelian group, $p\geq 2$ a prime, $a\in G$, 
and let $d$ be an integer coprime to $p$. Then $a$ is a $p$-th power in $G$ if and only if $a^d$ is a $p$-th power in $G$.
\end{lem}

\begin{proof}
The direct implication is obvious. Conversely, Bézout's Theorem asserts that there exist integers $u$ and $v$ such that $pu + dv = 1$.
The equality $a = (a^u)^p \cdot (a^d)^v$ shows that if $a^d$ is a $p$-th power, then so is $a$.

\end{proof}

The second lemma will notably be combined with Theorem \ref{BN} to construct more groups satisfying $\mathcal{P}(p)$ in the non-abelian case.

\begin{lem}\label{Imp}
Let $G$ be a group, $p$ a prime number, $x \in G$ of order $p^m$ with $m\geq 2$, and $D$ a subgroup of $G$. Let $\alpha=x^{p^{m-1}}$, then the following assertions are equivalent.
\begin{enumerate}
    \item There exists an integer $k\in\{1,\ldots,m-1\}$ such that $ \langle x^{p^{k}} \rangle \cap D= \{1\}$.
    \item $\alpha \notin D$.
    \item $\langle x \rangle \cap D = \{1\}$.
\end{enumerate}
\end{lem}

\begin{proof}
$1) \implies 2)$ We have $\alpha=(x^{p^k})^{p^{m-1-k}}\in \langle x^{p^k} \rangle\setminus\{1\}$, so $\alpha \notin D$.

$2) \implies 3)$ Suppose, by contradiction, that 
there exists $j \geq 1$ such that $x^j \in D \setminus \{1\}$. Let $r \geq 1$ be such that $x^j$ is of order $p^r$ (such an $r$ exists since the order of $x^j$ divides $p^m$). The cyclic $p$-groups $\langle x^j \rangle$ and $\langle x\rangle$ both contain a unique subgroup of order $p$, which is none other than $ \langle \alpha \rangle$. This contradicts $\alpha \notin D$.

$3) \implies 1)$ is trivial.

\end{proof}

A natural question that arises is whether we can reduce the study of Property $\mathcal{P}(p)$ to $p$-groups. The following proposition relates, for a given group, the property $\mathcal{P}(p)$ to the property $\mathcal{P}(p)$ for its $p$-Sylow subgroups.

\begin{prop}\label{Sylow}
  Let $G$ be a finite group, and let $S \subseteq G$ be a $2$-Sylow subgroup of $G$. Then one has
  \begin{enumerate}
      \item If $G$ is abelian and $p\geq2$ is a prime, then $G$ satisfies $\mathcal{P}(p)$ if and only if its (unique) $p$-Sylow subgroup $P$ satisfies $\mathcal{P}(p)$.
      \item If $S$ is normal in $G$ (\textit{e.g.} if $G$ is abelian), and if $S$ satisfies $\mathcal{P}$ relatively to some $a,b\in S$, then $G$ satisfies $\mathcal{P}$ relatively to $a,b$.
      \item If $S$ satisfies $\mathcal{P}$ relatively to some $a,b\in S$ and if all $G$-conjugates of $a$ that belong to $S$ are non-squares in $S$ (e.g. if $a\in \mathcal{Z}(G)$), then $G$ satisfies $\mathcal{P}$ relatively to 
      $a,b$. 
  \end{enumerate}
\end{prop}

\begin{proof}
\begin{enumerate}
    \item Let $a,b \in G$ be such that $G$ satisfies $\mathcal{P}(p)$ relatively to $a,b$. Set $d=|G|/|P|$, and consider $a'=a^d$ and $b'=b^d$. We have $a',b'\in P$ and since $p\nmid d$, Lemma~\ref{2} implies that $a'$ is not a $p$-th power in $G$ while $b'$ is a $p$-th power in $G$. By Lemma~\ref{1}, a $p$-th root of $b'$ has $p$ power order and thus lies in $P$. We conclude by Proposition~\ref{Psimp}.
    
    The converse is trivial since, by the Structure Theorem for finite abelian groups, $P$ is a direct factor of $G$. Thus, if $P$ satisfies $\mathcal{P}(p)$, then $G$ satisfies $\mathcal{P}(p)$ by Lemma~\ref{Direct-product}.
    \item By Proposition~\ref{Psimp} it is enough to show that $a$ is not a square in $G$. By contradiction, if $c \in G$ is such that $c^2 = a$, then $\Ord(c) = 2\Ord(a)$ (by Lemma~\ref{1}) is a power of $2$, so $c$ belongs to the unique $2$-Sylow subgroup $S$ of $G$, contradicting the fact that $a$ is not a square in $S$. 
    \item By contradiction, if $G$ does not satisfy $\mathcal{P}$ with respect to $a,b$ then $a$ is the square of some $c\in G$. However $a\in S$, and therefore $c$ has $2$-power order. By Sylow's Theorem, there exists $c'\in S$ which is a $G$-conjugate of $c$. Hence $c'^2$ is a square in $S$ and is a $G$-conjugate of $a$. This contradicts the assumption, thus  $G$ satisfies $\mathcal{P}$ relatively to $a,b$.
\end{enumerate}
\end{proof}

Keeping the notation as in Proposition~\ref{Sylow}, the fact that $S$ satisfies $\mathcal{P}$ relatively to $a,b$ does not imply in general that $G$ satisfies $\mathcal{P}$ relatively to $a,b$. This can be seen in the following example.
\begin{exe}\label{SD}
The subgroup $S:=\langle (1\ 2\ 3\ 4),(1\ 2)(3\ 4) \rangle \subset \frak{S}_4$ is a $2$-Sylow subgroup. If $a'=(1\ 2)(3\ 4)$ and $b'=(1\ 3)(2\ 4)$ then $S$ satisfies $\mathcal{P}$ with respect to $a',b'$. However, $a'$ and $b'$ are conjugates in $\frak{S}_4$, so for any choice of conjugates $a,b$ of $a',b'$ respectively, $\frak{S}_4$ does not satisfy $\mathcal{P}$ relatively to $a,b$. 
\end{exe}

Conversely, if $G$ satisfies $\mathcal{P}$ we cannot conclude in general that its $2$-Sylow subgroups satisfy $\mathcal{P}$, as the following example shows. 
\begin{exe}
    Consider the semi-direct product $\Z/6\Z \rtimes \Z/12\Z$ where the structure is defined as follows. Let $\Z/3\Z$ act trivially on $\Z/6\Z$ and let $\Z/4\Z$ act on $\Z/2\Z\times \Z/3\Z \simeq \Z/6\Z$ in such a way that elements of order $4$ of $\Z/4\Z$ act by inversion on $\Z/3\Z$ and trivially on $\Z/2\Z$. Let $\alpha_1$, $\alpha_2$ be the (unique) elements of order $2$ in $\Z/6\Z$ and $\Z/12\Z$ respectively. The subgroup $\langle \alpha_1,\alpha_2 \rangle$ is normal in  $(\Z/6\Z \rtimes \Z/12\Z$) and we consider $G:=(\Z/6\Z \rtimes \Z/12\Z)/\langle \alpha_1,\alpha_2 \rangle$. Then the $2$-Sylow subgroups of $G$ are isomorphic to $\Z/4\Z$ which does not satisfy $\mathcal{P}$, whereas $G$ satisfies $\mathcal{P}$ relatively to the classes of a generator of $\Z/6\Z$ and an element of order $6$ in $\Z/12\Z$.
    
\end{exe}

\subsection{Examples of groups enjoying $\mathcal{P}(d)$ and characterization for certain classes}\label{sec:exchar}

We start by focusing on $p$-groups. Our first characterisation relates Property $\mathcal{P}(p)$ to cyclic subgroups of $G$.  



\begin{prop}\label{Caractérisation}
Let $G$ be a $p$-group. Then $G$ satisfies $\mathcal{P}(p)$ if and only if $G$ has two maximal cyclic subgroups of distinct orders.
\end{prop}

\begin{proof}
Assume that $G$ satisfies $\mathcal{P}(p)$ with respect to $a,b$. Let $\langle c \rangle$ be a maximal cyclic subgroup of $G$ containing $b$. Writing $b=c^k$, for some $k\in \N$, we necessarily have $p\mid k$, since otherwise $b$ generates a maximal cyclic subgroup (namely $\langle c\rangle$) of $G$, contradicting the fact that $b$ is a $p$-th power in $G$. 
Therefore $\Ord(c) > \Ord(b) = \Ord(a)$. Moreover $\langle a \rangle $ is a maximal cyclic subgroup in $G$ (otherwise $a=c'^d$ for some $d\in\N_{\geq 2}$ and some $c'\in G$ of order strictly bigger that $\Ord(a)$, which implies that $p\mid d$, contradicting the fact that $a$ is not a $p$-th power), hence the result.

Conversely, suppose that $G$ has two maximal cyclic subgroups $\langle a \rangle $ and $\langle c \rangle $ with $\Ord(c) = p^m$ and $\Ord(a) = p^n$ for integers $1\leq n < m$.
The element $b=c^{p^{m-n}}$ has order $p^n$ and $b$ is the $p$-th power of $c^{p^{m-n-1}}$. By maximality of $\langle a\rangle$, the element $a$ is not a $p$-th power. We conclude by applying Proposition~\ref{Psimp}.

\end{proof}

As an illustration of families of groups satisfying Property $\mathcal P$, we recall the definition of a $Q$-group (a generalization of quaternion groups; see~\cite{BN}).
\begin{defi}
A finite $2$-group is a \emph{$Q$-group} if:
\begin{enumerate}
    \item $G$ has an abelian subgroup $A$ of index $2$ that has an element of order $>2$,
    \item $G=\langle A,x \rangle$ where $x\in G$ is of order $4$ and satisfies $xax^{-1}=a^{-1}$ for all $a \in A$.
\end{enumerate}
\end{defi}

For such a class of groups, we draw the following consequence of Proposition~\ref{Caractérisation}.
\begin{cor}\label{q-gr}
    Any $Q$-group that has an element of order $8$ satisfies Property $\mathcal{P}$.
\end{cor}
\begin{proof}
 Let $G$ be a group satisfying the assumption, let $A$ be a subgroup of $G$ of index $2$, and let $x\in G$ be of order $4$ such that $G=\langle A, x  \rangle $. Note that $x\notin A$. Every square in $G$ represents the trivial class in $G/A$, hence every square in $G$ belongs to $A$. In particular, $x$ is not a square. Thus, $\langle x \rangle $ is a maximal cyclic subgroup (by the same argument showing that $\langle a \rangle$ is maximal in the first part of the proof of Proposition~\ref{Caractérisation}) and it has order $4$. Since $G$ has an element of order $8$, then $G$ has two maximal cyclic subgroups of distinct orders. We conclude thanks to Proposition~\ref{Caractérisation}.

\end{proof}

The notion of $Q$-group was introduced by Blackburn in order to prove the following theorem (see~\cite{BN}*{Theorem 1}). Recall that a Dedekind group is a group such that every subgroup is normal. For a characterization of Dedekind groups one can see~\cite{Hall}*{Theorem 12.5.4}.
\begin{thm}[Blackburn]\label{BN}
    Let $G$ be a $p$-group that is not Dedekind and such that the intersection of all non-normal subgroups of $G$ is non-trivial. Then $p=2$ and one of the following cases occurs:
    \begin{enumerate}
        \item $G$ is the direct product of the quaternion group of order $8$, a cyclic group of order $4$, and an (abelian) elementary $2$-group (i.e., a group in which every element is of order $2$).
        \item $G$ is the direct product of two quaternion groups of order $8$ and an (abelian) elementary $2$-group.
        \item $G$ is a $Q$-group.
    \end{enumerate}
\end{thm}


Our next goal is to prove Theorem~\ref{Ord} below. For this matter Blackburn's Theorem will be useful. Theorem~\ref{Ord} provides simple conditions for a $p$-group to satisfy $\mathcal P(p)$.

\begin{thm}\label{Ord}
    Let $G$ be a non-abelian $p$-group, where $p$ is a prime.
    \begin{enumerate} 
    \item Assume that $|G|=p^{m+n}$ for integers $1\leq n <m$ and that $G$ is not isomorphic to the quaternion group of order $8$. Suppose further that there exists $y\in G$ of order $p^m$. Then $G$ satisfies $\mathcal P(p)$.
    \item  Assume that $G$ contains an element of order at least $\max(p^2,8)$.  Then one of the following cases occurs:
    \begin{enumerate}
        \item $G$ satisfies property $\mathcal{P}(p)$.
        \item If $\alpha \in \mathcal{Z}(G)$ is an element of order $p$ then $G/\langle \alpha \rangle$ satisfies property $\mathcal{P}(p)$.
    \end{enumerate}
    Consequently, the group $\overline{G}:=G\times (G/\langle \alpha \rangle)$ satisfies property $\mathcal{P}(p)$.
    \end{enumerate}
    \end{thm}

Before giving the proof, we briefly sketch the strategy. Theorem~\ref{Ord}(1) states that when a $p$-group $G$ has an element $y$ with large order, then $G$ satisfies $\mathcal{P}(p)$. Thanks to Proposition \ref{Caractérisation}, a $p$-group satisfies $\mathcal{P}$ as soon as there exist two maximal cyclic subgroups of distinct orders. To prove Theorem \ref{Ord} we will choose an element $a$ in $G$ such that $\langle a \rangle \cap \langle y \rangle $ is as small as possible and note that since $y$ has ``large'' order, then this implies that $a$ should have a ``small'' order, otherwise $|\langle a,y \rangle |\geq |\langle a \rangle.\langle y \rangle | $ would be larger than $|G|$. To prove the existence of such an element $a$, we will combine Lemma \ref{Imp} with Blackburn's Theorem~\ref{BN}.

\medskip

To begin with, we state the following lemma (see e.g.~\cite{Berkovich}*{Introduction, exercise 10}) that will be used several times and in particular in the proof of Theorem~\ref{Ord}. 
\begin{lem}\label{card}
Let $G$ be a group, and $H,K$ be two finite subgroups. Then, denoting $HK=\{hk\colon h\in H\text{ and } k\in K\}$, we have
$$|HK|=\frac{|H|\cdot|K|}{|H \cap K|}\,.$$
\end{lem}

\begin{proof}[Proof of Theorem~\ref{Ord}]
(1) Note that $|G|\geq 8$. We consider two cases separately.

First assume $|G|=8$. Since $G$ is non-abelian and is not the quaternion group. We conclude that $G$ is isomorphic to the dihedral group $D_8$ of order $8$ which satisfies $\mathcal{P}$ by Corollary~\ref{Example2}.

Next assume $|G|>8$.  Two cases can then occur:
\begin{enumerate}
    \item[(i)] There exists $z \in G$ non-trivial such that $\langle y \rangle \cap \langle z \rangle = \{1\}$. 
    
    Let $x \in G$ such that $\langle x \rangle$ is a maximal cyclic subgroup of $G$ containing $z$. Denote $p^r=\Ord(x)$. According to Lemma \ref{Imp} (applied for $D=\langle y\rangle$), we have $\langle x \rangle \cap \langle y \rangle =\{1\}$. Then $|\langle x \rangle \langle y \rangle |=p^{m+r}\leq p^{m+n}$, implying $r\leq n < m$, and therefore $G$ satisfies $\mathcal{P}(p)$, by Proposition \ref{Caractérisation},.
    
    \item[(ii)] For every non-trivial $z\in G$, we have $\langle y \rangle \cap \langle z \rangle \ne \{1\}$.
    
    Denote $\alpha=y^{p^{m-1}}$. Applying Lemma~\ref{Imp}, we deduce that for every non-trivial $z\in G$, we have $\alpha \in \langle z \rangle$. In particular the intersection of the non-trivial subgroups of $G$ is non-trivial (since it contains $\alpha$). 
    
    Assume that $p$ is odd. Then $G$ is not Dedekind, since Dedekind $p$-groups (for odd $p$) are abelian (see~\cite{Hall}*{Theorem 12.5.4}), and we get a contradiction by combining Theorem~\ref{BN} and what we have just proved. We deduce that $p=2$. Consequently $G$ is not Dedekind: indeed $m>2$ (by the assumptions $1\leq n<m$ and $|G|=p^{n+m}>8$) and therefore $\langle y\rangle$ has a subgroup of order $8$. We conclude by invoking \cite{Hall}*{Theorem 12.5.4} implying that a non-abelian Dedekind group cannot contain an element of order $8$. 
    
   Theorem \ref{BN} combined with the fact that $G$ has an element of order $8$ implies that $G$ is a $Q$-group. Invoking Corollary \ref{q-gr}, the proof is complete. 
\end{enumerate}

(2) Since $G$ is a non-abelian group (containing an element of order $8$ when $p=2$), then $G$ is not a Dedekind group (again by \cite{Hall}*{Theorem 12.5.4}). Let $\alpha\in\mathcal Z(G)$ be  an element of order $p$ and let $ \langle y \rangle$ be a maximal cyclic subgroup of order $p^m$ such that $y^{p^{m-1}}=\alpha$ (note that any maximal cyclic subgroup of $G$ that contains $\alpha$ is generated by such an element $y$). If $m=1$ or if $p=m=2$, then according to Proposition~\ref{Caractérisation} $G$ satisfies $\mathcal{P}(p)$. Otherwise, two cases arise:
\begin{enumerate}
    \item[(i)] Suppose there exists $x \in G \setminus \{1\}$ such that $\langle x \rangle \cap \langle y \rangle = \{1\}$. Up to replacing $\langle x\rangle$ by a maximal cyclic subgroup containing $x$ and intersecting $\langle y \rangle$ trivially (by Lemma \ref{Imp}), we can assume that $\langle x \rangle$ is a maximal cyclic subgroup. Denote $\Ord(x)=p^n$. Proposition~\ref{Caractérisation} shows that if $n\ne m$, then $G$ satisfies $\mathcal{P}$. Suppose $n=m$. The class $\overline{x}$ of $x$ in $G/\langle \alpha \rangle$ has order $p^m=p^n$ (recall that $\langle \alpha\rangle\cap\langle x \rangle=\{1\}$). It is sufficient to show that $\overline{y}$ generates a maximal cyclic subgroup in $G/\langle \alpha \rangle$. Indeed this implies that $\langle \overline{y}\rangle$ and $\langle \overline{x}\rangle$ are maximal subgroups of distinct orders in $G/\langle \alpha\rangle$ and we conclude by applying Proposition~\ref{Caractérisation}. 
    
    Assume by contradiction that there exists $z \in G$ such that $\overline{z}^p= \overline{y}$. We then have $z^py^{-1} \in \{\alpha^k\colon 0\leq k < p\}$. Every $y\alpha^k=y^{kp^{m-1}+1}$ with $0\leq k<p$, generates $\langle y \rangle$ (since the exponent $kp^{m-1}+1$ is coprime to $p$). Thus $\langle z^p\rangle$ is a maximal cyclic subgroup, 
    leading to the desired contradiction. 
    
    \item[(ii)] Suppose that for every non-trivial $x \in G$, $\langle x \rangle \cap \langle y \rangle \ne \{1\}$. According to Lemma \ref{Imp}, we have $\alpha \in \langle x\rangle$ for every $x \in G \setminus \{1\}$. Thus we can apply Theorem \ref{BN}. Since $G$ has an element of order $8$ we deduce that $G$ is a $Q$-group and in turn that $G$ satisfies $\mathcal{P}$ by Corollary \ref{q-gr}. The last part of the statement follows by Lemma~\ref{Direct-product}.
\end{enumerate}

\end{proof}

The abelian case is much easier. Thanks to the Structure Theorem of finite abelian groups, we can give a full characterization of abelian groups $G$ satisfying $\mathcal{P}(p)$ for $p$ prime.
\begin{thm}\label{Ab}
    Let $G$ be a finite abelian group and let $p$ be a prime number. Then $G$ satisfies $\mathcal{P}(p)$ if and only if its (unique) $p$-Sylow subgroup is not homocyclic.
\end{thm}
\begin{proof}
By Proposition~\ref{Sylow}(1), it is enough to prove the theorem when $G$ is an abelian $p$-group. If $G$ is not homocyclic, then $G$ has two cyclic direct factors of distinct orders. Since these direct factors are maximal cyclic subgroups, then, by Proposition~\ref{Caractérisation} $G$ satisfies $\mathcal{P}(p)$.\\
Conversely, if $G$ satisfies $\mathcal{P}(p)$, then by Proposition \ref{Caractérisation} $G$ has two maximal cyclic subgroups with distinct orders. We use the well-known fact that if $\langle h \rangle$ is a maximal cyclic subgroup of $G$, then $G$ has a direct factor of order $\Ord(h)$ (otherwise, applying the Structure Theorem for finite abelian groups, $h$ would have a $p$-th root which would have order $p \cdot \Ord(h)$ by Lemma~\ref{1} contradicting the maximality of $\langle h \rangle$). Thus, $G$ has two cyclic direct factors with distinct orders which means that $G$ is not homocyclic.

\end{proof}

\begin{proof}[Proof of Theorem~\ref{Bias-Ab-gr}.]
Assume $L/K$ does not have any extreme Chebyshev bias relatively to any conjugacy classes $C_1,C_2$ in $G:=\Gal(L/k)$). Then by Proposition~\ref{prop:PimpliesBias}, $G$ does not satisfy $\mathcal{P}(p)$ for any prime $p$. In particular, Theorem~\ref{Ab} implies that for every $p$, the $p$-Sylow subgroup of $G$ is homocyclic.

Conversely assume that for every $p$, the $p$-Sylow subgroup of $G$ is homocyclic. Then, by Theorem \ref{Ab}, $G$ does not satisfy $\mathcal{P}(p)$ for any prime $p$ dividing $|G|$. For such a $p$, since $f_p\colon G\rightarrow G$ (defined by $f_p(x)=x^p$) is a group homomorphism, every $p$-th power in $G$ has the same number of $p$-th roots. So that, every two elements of the same order in $G$ have the same number of $p$-th roots (otherwise $G$ would satisfy $\mathcal{P}(p)$). \\
Let $a,b \in G$ be two elements of the same order. Let us show that for every square-free $d\geq 2$, we have $r_d(a)=r_d(b)$. Since this equality is true for primes, it is sufficient to prove that the map $m\mapsto r_m(a)$ is multiplicative, that is for coprime $m,n\geq 1$ we have : $$r_{mn}(a)=r_{m}(a)r_n(a)$$
Denote for $k\geq1$: $$\Gamma_k(a):=\{x\in G\ :\ x^k =a \}$$
By Bézout's theorem, there exists $u,v \in \Z$ such that $um+vn=1$.\\
Let's prove that, the map $\Gamma_{mn}(a)\rightarrow \Gamma_m(a)\times \Gamma_n(a)$ given by $x\rightarrow (x^n,x^m)$ is bijective. The injectivity follows directly from the equality $x^{um}.x^{vn}=x$ which is true for all $x\in G$. While, if $(y,z)\in \Gamma_m(a)\times \Gamma_n(a)$. Denote $x=y^vz^u \in G$, the commutativity of $G$ shows that $$x^{mn}=(y^m)^{nv}(z^n)^{um}=a^{vn}a^{um}=a$$
Thus $x\in \Gamma_{mn}(a)$. Moreover, $x^n=y^{nv}a^u=y^{nv}y^{mu}=y$ and samewise $x^m=z$. This proves that for all $d\geq 2$ square-free we have $r_d(a)=r_d(b)$, hence by Theorem~\ref{OC-thm}(1) we have $$\pi(x;L/K;\{a\})=\pi(x;L/K;\{b\})\,.$$ Thus there is no Chebyshev bias with respect to $\{a\}, \{b\}$.

\end{proof}

\section{Extensions containing subextensions with extreme Chebyshev bias}\label{section:P+}

So far we have been focusing on Cayley embeddings (combined with Property $\mathcal P(d)$ studied in \S\ref{Property_P}; recall Definition~\ref{P}) to uncover a variety of cases where extreme Chebyshev biases can be obtained in Galois extensions of number fields.

In this section our goal is to relax the constraint on the index $(G^+:G)$ that is as large as $(|G|-1)!$ in the case of the Cayley embedding and obtain other examples of extensions $L/K$ of number fields with extreme Chebyshev biases. To do so, we introduce the following property which can be seen as a version of $\mathcal P(d)$ on the level of the group $G^+$.

\begin{defi}[Property $\mathcal P^+(p)$]\label{P+}
Let $G^+$ be a group, let $a,b\in G^+$, and let $p$ be a prime number. One says that $G^+$ satisfies $\mathcal{P^+}(p)$ if there exist conjugate elements $a,b\in G^+$ and a subgroup $G \subset G^+$ containing $a,b$ such that $G$ satisfies $\mathcal{P}(p)$ relatively to $(a,b)$.
\end{defi}

As in the case of the property $\mathcal{P}$, to simplify notations when $p=2$, we will simply write ``$G^+$ satisfies $\mathcal P^+$''. 


The following statement makes precise the way in which one can use $\mathcal P^+(p)$ to obtain extreme Chebyshev biases.

\begin{prop}\label{P+simp}
    Let $G^+$ be a group satisfying property $\mathcal{P^+}(p)$ relatively to $(a,b)$ and $G\subset G^+$ a subgroup containing $a,b$ such that $G$ satisfies $\mathcal{P}(p)$ relatively to $(a,b)$. Let $L/k $ be a Galois extension of number fields with group $G^+$, and denote $K=L^G$. Then, if $G^+$ is a $p$-group or if $p=2$, there exists $A>0$ such that for all $x\geq A$,
    \[
    \frac{|G|}{|C_a|}\pi(x;L/K;C_a)>\frac{|G|}{|C_b|}\pi(x;L/K;C_b)\,.
    \] 
\end{prop}
\begin{proof}
We show that $t:=t_{C_a,C_b}$ satisfies the assumptions of Proposition \ref{gen} for $d=p$. Since $G$ satisfies property $\mathcal{P}(p)$ relatively to $(a,b)$, then $\langle t,r_p\rangle <0$. As $G^+$ satisfies $\mathcal{P^+}(p)$ relatively to $(a,b)$, we have $t^+=0$. It remains to show that for all $1< \ell < p$ square-free, $(t\circ f_\ell)^+=0$. If $p=2$, there is nothing to prove. If $G^+$ is a $p$-group for $p>2$, every $1<\ell<p$ is coprime with $p$. Therefore, the function $f_\ell:G^+\rightarrow G^+$ defined by $f_\ell(g)=g^\ell$ for $g\in G^+$ satisfies $f_\ell(G)=G$, and thus for all $g\in G^+$, we have $g\in G$ if and only if $f_\ell(g)\in G$. Then, for all $x\in G^+$:
\begin{equation}\label{eq:indnulle}
(t\circ f_\ell)^+(x)=\frac{1}{|G|}\underset{\substack{g\in G^+\\gxg^{-1}\in G}}{\sum}(t\circ f_\ell)(gxg^{-1})=\frac{1}{|G|}\underset{\substack{g\in G^+\\gx^\ell g^{-1}\in G}}{\sum}t(gx^\ell g^{-1})=t^+\circ f_\ell (x)=0\,.
\end{equation}

\end{proof}

\begin{rk}
   The assumption ``$p=2$ or $G^+$ is a $p$-group'' in Proposition~\ref{P+simp} cannot be relaxed in general. Indeed extreme Chebyshev biases are produced using Proposition~\ref{gen}, where crucially $(t\circ f_\ell)^+=0$ for all $1\leq \ell<p$. As in Remark~\ref{rem:converse}, let $i,j$ be two distinct elements of order $4$ generating the quaternion group $Q_8$. Let $C_{p^2}=\langle c \rangle$ be a cyclic group of order $p^2$, and let $C_p=\langle c_p \rangle$ be its subgroup of order $p$. Consider $G:=\left(\langle i \rangle \times C_p \right) \times \left (\langle j \rangle \times C_{p^2}\right)$, and let $a:=(-1,c_p,1,1)$ and $b:=(1,1,-1,c_p)$. The group $G$ satisfies $\mathcal{P}(p)$ relatively to $(a,b)$. One has the diagonal (natural) injection $G\hookrightarrow (Q_8\times C_{p^2}) \times (Q_8\times C_{p^2})$ which is itself a subgroup of $G^+:=\left((Q_8\times C_{p^2}) \times (Q_8\times C_{p^2})\right)\rtimes \Z/2\Z$, the wreath product of $Q_8\times C_{p^2}$ by $\Z/2\Z$. Denoting $t=\mathds{1}_{\{a\}}-\mathds{1}_{\{b\}}$, for all $1\leq \ell <p$ we have $\langle t, r_\ell \rangle_G=0$ and $t^+=0$. Thus the group $G^+$ satisfies $\mathcal{P}^+(p)$ relatively to $\left((a,\Bar{0}),(b,\Bar{0})\right)$. However $(t \circ f_2)^+\ne 0$, thus Proposition \ref{gen} does not yield an extreme Chebyshev bias in this case.
\end{rk}

As explained above, if $G$ satisfies $\mathcal{P}(p)$ relatively to $(a,b)$ then $\frak{S}(G)$ satisfies $\mathcal{P^+}(p)$ (\emph{via} the Cayley embedding $G\hookrightarrow \mathfrak S(G)$). We now address the question of the minimal order of a group $G^+$ satisfying $\mathcal{P}^+(p)$ when $p$ is a prime. We first note that if $G^+$ satisfies $\mathcal{P^+}(p)$ relatively to $(x,y)$ and $G$ is a subgroup of $G^+$ satisfying $\mathcal{P}(p)$ relatively to $(x,y)$, then $G$ is not normal in $G^+$ (otherwise an element of $G^+$ conjugating $x$ and $y$ would give rise to an automorphism $f$ of $G$ such that $f(y)=x$ which is impossible since, if $z$ is a $p$-th root of $y$, $x$ would have $f(z)$ as a  $p$-th root in $G$). So a first (trivial) answer to our question is that $|G^+|\geq 3|G|$ (since a subgroup of index $2$ is normal). As a corollary of Proposition~\ref{prop:p3} we have the following bound.

\begin{cor}\label{prop:3p3}
    Let $p$ be a prime and let \(G^+\) be a group satisfying $\mathcal{P^+}(p)$. Then, $|G^+| \geq 3p^3$.
\end{cor}
When $G^+$ satisfies $\mathcal{P}^+(p)$ relatively to a group $G$ of the form described in Theorems \ref{Ord}(1) and \ref{Ab}, Propositions \ref{min-ab} and \ref{min-nonab} will give a more precise answer to this minimality question.


 Our next goal will be to give a complete characterization of $p$-groups $G^+$ satisfying $\mathcal{P}^+(p)$ (see Theorem~\ref{CarP+}) which, combined with Proposition~\ref{P+simp}, will lead to the proof of Theorem \ref{p_groupsBias}.

\subsection{Minimal order of $G^+$ satisfying $\mathcal P^+(p)$}

Let $G^+$ satisfy $\mathcal P^+(p)$ and let $G$ be the corresponding subgroup. We start with the case where $G$ is abelian. For convenience we will use the following multiplicative notation.

Fix $p$ a prime and let $(E_n)_{n \in \mathbb{N}}$ be an increasing sequence of multiplicative cyclic groups such that $E_n$ has order $p^n$. For each $n\in \mathbb{N}$, let $a_n$ be a generator of $E_n$ such that $a_{n+1}^p=a_n$ (one can think of $E_n$ as the group of $p^n$-th roots of unity in $\C$ and choose $a_n=\exp(2i\pi\cdot p^{-n})$).

Theorem \ref{Ab} shows that we can assume $G=E_n \times E_m \times H$, where $1\leq n < m$ and $H$ is a finite abelian group. The minimal possible index $(G^+:G)$ in this situation is $2p$ as we now show by proving first Proposition~\ref{G+ab} and next Proposition~\ref{min-ab}.

\begin{proof}[Proof of Proposition~\ref{G+ab}]

Let $\overline{G}=E_m\times E_m \times H$, and consider $f \in {\rm Aut}(\overline{G})$ defined by $f(x,y,z)=(y,x,z)$ for $(x,y,z) \in \overline{G}$. The order of $f$ in the group ${\rm Aut}(\overline{G})$ is $2$.

Let $G^+:=\overline{G} \rtimes \Z/2\Z$, where the semi-direct product is given by the morphism 
\[\varphi\colon \Z/2\Z\to {\rm Aut}(\overline{G})\,,\qquad \bar 1\mapsto f\,.
\]

We have an injective (canonical) homomorphism $\iota \colon G \rightarrow S $ by composing the inclusion $G \subset \overline{G}$ with the canonical homomorphism $\overline{G} \rightarrow \overline{G} \rtimes \Z/2\Z=G^+$.  

Taking $a=(a_n,1,1)\in G$ and $b=(1,a_n,1)\in G$, we observe that $a$ is not a square in $G$ and that $b$ is a square in $G$. Moreover, if we denote $t \in \Z/2\Z$ as the non-trivial element (\textit{i.e.} the class of $1\in \Z$), and $1_G$ as the identity element of $G$, then we have $(1_G,t)\iota (a)(1_G,t)^{-1}=\iota (b)$. This implies that $G^+$ satisfies property $\mathcal{P^+}(p)$ relatively to $(\iota(a),\iota(b))$.
 Noting that the inverse Galois problem for the solvable group $G^+$ is solved over any number field $k$, this finishes the proof of Proposition \ref{G+ab}.
 \end{proof}

As we now show, it turns out that the group $(E_m\times E_m \times H)\rtimes \Z/2\Z$ has minimal order among groups $G^+$ satisfying $\mathcal{P^+}(p)$ and containing $G=E_n \times E_m \times H$.

\begin{prop}\label{min-ab}
With notation as above, let $G=E_n \times E_m \times H$. Suppose $G^+$ is a finite group containing $G$ as a subgroup and satisfying $\mathcal{P^+}(p)$, then one has
$$ |G^+| \geq 2p^{2m}|H|=| (E_m \times E_m\times H) \rtimes \Z / 2 \Z  |\,.$$
\end{prop}
\begin{proof}
Let $y=(1,a_m,1)$ and let $t\in G^+$ be such that $t(1,a_n,1)t^{-1}=(a_n,1,1)$. Define $x=tyt^{-1}$, then $x$ and $y$ both have order $p^m$.

Let $D=\{1\}\times E_m \times H$, which is a subgroup of $G$ and hence also a subgroup of $G^+$. We have
$$ x^{p^{m-n}}=t y^{p^{m-n}} t^{-1}=t (1,a_n,1)t^{-1}=(a_n,1,1)\,.$$

This implies $\langle x^{p^{m-n}}\rangle\cap D = \{1\}$, but then $\langle x \rangle \cap D=\{ 1 \}$ by Lemma~\ref{Imp}. Thus, we have an injective map (not necessarily a group morphism):
\[\theta \colon \langle x \rangle \times D \rightarrow \langle \{ x \} \cup D \rangle\,,\qquad  \theta(a,b)=ab\,.\]
Therefore we have $|\langle \{ x \} \cup D \rangle| \geq |\langle x \rangle | \times |D|=p^{2m}\times |H|$.
Moreover, since $x^{{p}^{m-n}}=(a_n,1,1) \in G $ and $G$ is commutative, the element $x^{p^{m-n}}$ commutes with the elements of the subgroup $\langle \{x\} \cup D \rangle$ (as it commutes with the generators of this subgroup).

However $x^{p^{m-n}}$ does not commute with $t$ (since $t(1,a_n,1)t^{-1}=(a_n,1,1)=x^{{p}^{m-n}}$). Therefore, $t \notin  \langle \{x\}  \cup D \rangle $. Hence, $\langle \{x \} \cup D \rangle$ is a proper subgroup of $G^+$ and has an index at least $2$ in $G^+$.
We conclude that $$|G^+| \geq 2 |\langle \langle x \rangle \cup D \rangle| \geq 2p^{2m} \times |H|.$$

\end{proof}

Next we work in the context of Theorem~\ref{Ord}(1): $G$ is a non-abelian group of order $p^{m+n}$, with $1\leq n < m$, that is not a $Q$-group (in particular, $G$ is not the quaternion group of order $8$). We assume that $G$ has an element $y$ of order $p^m$.

\begin{prop}\label{min-nonab}
    With notation as above, there exists $(a,b) \in G \times \langle y \rangle $, with $\langle a \rangle \cap \langle b \rangle = \{1\}$, such that $G$ satisfies $\mathcal{P}(p)$ relatively to $(a,b)$. Moreover, if $G$ is a subgroup of a group $G^+$ satisfying $\mathcal{P^+}(p)$ relatively to $(a,b)$, then $|G^+| \geq p^{m+n}(p^{m-n}+1)$.
\end{prop}
\begin{proof}
Since $G$ is not a $Q$-group, then, the proof of Theorem~\ref{Ord}(1), proves the existence of such elements $a,b \in G$.\\
Let $t \in G^+ $ be such that $tbt^{-1}=a$. Setting $x=tyt^{-1}$, Lemma \ref{Imp} shows that $\langle x \rangle \cap \langle y \rangle = \{1\}$, therefore $|\langle x \rangle \langle y \rangle | =|\langle x \rangle |.|\langle y \rangle |=p^{2m}$.\\
We claim that $t \notin \langle x \rangle \langle y \rangle$: by contradiction let $i,j$ be integers such that $t=x^i y^j$. Then we have $x=tyt^{-1}=x^i y x^{-i}$ from which we deduce that $y\in \langle x \rangle$; a contradiction.
Therefore, $|G^+| > |\langle x \rangle \langle y \rangle |=p^{2m}$. As $|G|=p^{m+n}$ divides  $|G^+|$ then $$|G^+| \geq p^{2m}+p^{m+n}=p^{m+n}(p^{m-n}+1)\, .$$

\end{proof}

The lower bound in Proposition~\ref{min-nonab} is optimal as the following example shows.
\begin{exe}\label{S4}
    Let $G:=\langle (1\ 2\ 3\ 4),(1\ 2)(3\ 4) \rangle \subset \frak{S}_4=: G^+$ . Then $G$ is isomorphic to the dihedral group $D_8$ of order $8=2^{2+1}$. If $a=(1\ 2)(3\ 4)$ and $b=(1\ 3)(2\ 4)$, then $G$ satisfies $\mathcal{P}$ relatively to $(a,b)$. Also, $a$ and $b$ are conjugate in $G^+$, so $G^+$ satisfies $\mathcal{P^+}$ relatively to $(a,b)$. Moreover, $|G^+|=24=2^{2+1}(2^{2-1}+1)$ attaining the lower bound of Proposition~\ref{min-nonab}.
\end{exe}

The example shows that $\mathfrak S_4$ satisfies $\mathcal P^+$. In fact it has minimal order for Property $\mathcal P^+(p)$, uniformly over primes $p$ (this is a consequence of Corollary~\ref{prop:3p3}).


\subsection{Characterization of $p$-groups satisfying $\mathcal{P^+}$}

\begin{prop}\label{Inj}
    Let $H\rightarrow G^+$ be an injective group homomorphism. If $H$ satisfies $\mathcal{P^+}(p)$, then $G^+$ satisfies $\mathcal{P^+}(p)$. 
\end{prop}
\begin{proof}
$G^+$ contains a subgroup that satisfies $\mathcal{P^+}(p)$. Therefore, it contains a subgroup $G$ and two elements $a, b$ such that $G$ satisfies $\mathcal{P}(p)$ relatively $(a,b)$ and such that $a$ and $b$ are conjugate in a subgroup of $G^+$, hence conjugates in $G^+$. Thus, $G^+$ satisfies $\mathcal{P^+}(p)$. 

\end{proof}

\begin{cor}
    The symmetric group $\frak{S}_n$ satisfies $\mathcal{P}^+$ if and only if $n \geq 4$. The alternating group $\frak{A}_n$ satisfies $\mathcal{P}^+$ if and only if $n \geq 6$. 
\end{cor}

\begin{proof}
    Example \ref{S4} shows that $\frak{S}_4$ satisfies $\mathcal{P^+}$. Thus, using Proposition \ref{Inj}, $\frak{S}_n$ satisfies $\mathcal{P^+}$ for all $n \geq 4$. By Proposition $\ref{prop:3p3}$ $\frak{S}_n$ does not satisfy $\mathcal{P}^+$ for $n\leq 3$.
    
    As for the alternating group, Proposition \ref{Inj} allows us to only consider the case $n=6$. The subgroup $S:=\langle (1\ 2\ 3\ 4)(5\ 6),\ (1\ 2)(3\ 4) \rangle \subset \mathcal{A}_6$ has order $8$. The elements $a=(1\ 2)(3\ 4)$ and $b=(1\ 3)(2\ 4)$ have the same order, and $b$ is a square in $S$ while $a$ is not. Moreover, denoting $\sigma=(2\ 3)(5\ 6)\in \mathcal{A}_6$ we have $\sigma a \sigma^{-1}=b$.
    
    To see that $\frak{A}_n$ does not satisfy $\mathcal{P}^+$ for $n\leq 5$, it is sufficient to see that any of its $2$-Sylow subgroups (which is isomorphic to the Klein four group for $4\leq n \leq 5$ and trivial for $n\leq 3$) do not contain any element of order $4$ (here we argue as in the beginning of the proof of Proposition~\ref{prop:p3}).

\end{proof}

\begin{rk} Let $k$ be a finite field, using Proposition \ref{Inj} and using a faithful representation of $\frak{S}_n$ into $\GL_n(k)$ and $\text{PGL}_n(k)$ we see also that $\GL_n(k)$ and $\text{PGL}_n(k)$ satisfy $\mathcal{P}^+$ for $n\geq 4$. Samewise, $\mathcal{A}_n$ admits a faithful representation into $\text{SL}_n(k)$ and $\text{PSL}_n(k)$ and therefore these groups satisfy $\mathcal{P}^+$ when $n\geq6$.
\end{rk}

Proposition \ref{Inj} provides a tool to demonstrate that a group $G^+$ satisfies $\mathcal{P^+}(p)$. If we prove that a subgroup of $G^+$ (especially one of its $p$-subgroups) satisfies $\mathcal{P^+}(p)$, we can conclude that $G^+$ satisfies $\mathcal{P^+}(p)$. 

Theorem \ref{CarP+} gives a complete characterization of $p$-groups satisfying $\mathcal{P^+}(p)$.

\begin{thm}\label{CarP+}
    Let $G^+$ be a finite $p$-group. Then $G^+$ satisfies $\mathcal{P}^+(p)$ if and only if there exists $x\in G^+$ such that $\langle x^p \rangle$ is not normal in $G^+$.
\end{thm}
\begin{proof}
Let $G^+$ be a $p$-group satisfying $\mathcal{P^+}(p)$ relatively to a subgroup $G$ and a couple if elements $(a,b)$. Since $a$ is conjugate to $b$ in $G^+$ and since $b$ has $p$-th roots in $G$, there exists $x\in G^+$ be such that $x^p=a$. Note that $\langle a \rangle \neq \langle b \rangle$ (otherwise any subgroup containing a $p$-th root of $b$ must contain a $p$-th root of $a$), therefore if $t\in G^+$ satisfies $tat^{-1}=b$, then $t\langle a \rangle t^{-1}=\langle b \rangle \neq \langle a \rangle$, thus $\langle a\rangle=\langle x^p \rangle$ is not normal in $G^+$.

For the converse, we will need the following result.

\begin{lem}\label{ConjNorm}
    Let $G^+$ be a finite nilpotent group. Suppose that there exists $a\in G^+$ such that $\langle a \rangle$ is not normal in $G^+$. Then there exists a conjugate $b$ of $a$ in $G^+$ such that $b\in N_{G^+}(\langle a \rangle) \setminus \langle a \rangle $, where $N_{G^+}(\langle a \rangle)$ denotes the normalizer of the subgroup $\langle a \rangle$ in $G^+$.
\end{lem}

Using the Lemma we finish the proof of the Theorem.
Since $G^+$ is a $p$-group (and thus a nilpotent group), we can apply the lemma for $a=x^p$. There exists $b\in G^+$, a conjugate of $a$, such that $b \in N_{G^+}(\langle a \rangle) \setminus \langle a \rangle$.
Assume by contradiction that $G^+$ does not satisfy $\mathcal{P^+}(p)$. The subgroup $\langle x,b \rangle$ contains a $p$-th root $y$ of $b$ (otherwise, by Proposition~\ref{Inj} $G^+$ would satisfy $\mathcal{P^+}(p)$). Let $G=\langle a,y \rangle$ and note that since $y \in \langle x,b \rangle$ and $x$ and $b$ normalize $\langle a \rangle$, the subgroup $\langle a \rangle$ is normal in $G$. In particular, $\langle a \rangle \langle y \rangle=\langle y \rangle \langle a \rangle=G$ 
and the quotient group $G/\langle a \rangle$ is cyclic generated by the class of $y$. Let $\pi \colon G \rightarrow G/\langle a \rangle$ be the associated canonical projection. 

The subgroup $G$ contains $a$ and $b$ as well as a $p$-th root (namely $y$) of $b$, and therefore our assumption that $G^+$ does not satisfy $\mathcal{P^+}(p)$ implies that $G$ contains a $p$-th root $z$ of $a$.
The element $\pi(z)$ has order $p$, and $b\notin \langle a \rangle$ (recall that $b\in N_{G^+}(\langle a \rangle) \setminus \langle a \rangle$), so $\pi(b)$ has order $\geq p$ (since $G/\langle a\rangle$ is $p$-group as well) and thus $\Ord(\pi(y))=p^r$ for some $r\geq 2$. Also note that $y^{p^r}\in \langle  a \rangle$ has order $\leq \frac{\Ord(a)}{p}$ (since $y^p=b$ has the same order as $a$), so $y^{p^r} \in \langle a^p \rangle$. However, $\pi(z)=\pi(y)^{p^{r-1}s}$, where $s$ is coprime to $p$ (since $\pi(z)$ has order $p$ in $G/\langle a \rangle=\langle \pi(y)\rangle$). Thus, there exists $d\in \mathbb{Z}$ such that $z=y^{sp^{r-1}}a^d$. Therefore, $(za^{-d})^p=y^{sp^r} \in \langle a^p \rangle $. But $a$ commutes with $z$ (since $a\in \langle z \rangle$), so $a=z^p\in a^{pd}\langle a^p \rangle=\langle a^p \rangle$. This is a contradiction.

\end{proof}

Let's prove the lemma: \\
\begin{proof}[Proof of Lemma~\ref{ConjNorm}]
Assume, by contradiction, that there is no such element $b$. Thus, for every conjugate $y$ of $a$, we have: $$ y \notin \langle a \rangle \implies y \notin N_{G^+}\langle a \rangle\,.$$
We start by noticing that if $x, y \in G^+$, then $[x,y]y=xyx^{-1}$ is a conjugate of $y$ (here $[x,y]=xyx^{-1}y^{-1}$). Since $\langle a\rangle$ is not normal we consider an element $t \in G^+$ such that $t\langle a \rangle t^{-1} \ne \langle a \rangle $, and we set $\beta:=tat^{-1}$. We have $\beta \notin \langle a \rangle$, and thus $\beta \notin N_{G^+} \langle a \rangle$. Define the sequence $(b_n)_{n\geq 1}$ by 
\[
b_1:= [\beta,a]a=\beta a \beta ^{-1}\,,\qquad  b_{n+1}=[b_na^{-1},a]a\qquad (n \geq 1)\,.
\]
For every $n \geq 1$, the element $b_n$ is conjugate to $a$. Moreover $b_1 \notin \langle a \rangle$ since $\beta \notin N_{G^+} \langle a \rangle$. If $b_n \notin \langle a \rangle$ for some $n\geq 1$, then $b_n \notin N_{G^+}\langle a \rangle $, and as $$b_{n+1}=b_na^{-1}\cdot a\cdot a\cdot b_n^{-1}a^{-1}a=b_nab_n^{-1}\,,$$
then $b_{n+1} \notin \langle a \rangle$. This induction shows that $b_n\notin\langle a\rangle$ for all $n\geq 1$. 

In particular, we have $b_n \ne a$ for every $n \geq 1$. However the equality $b_n a^{-1}= [b_{n-1}a^{-1},a]$ implies that for every $n \geq 1$, 
$$b_n a^{-1}=[\ldots [[\beta,a],a],\dots,a]
$$
where $a$ appears $n$ times.
This shows that $b_na^{-1}\in C^{(n+1)}(G^+)$, where $(C^{(n)}(G^+))_{n\geq 1}$ is the descending central series of $G^+$. Since $G^+$ is nilpotent there exists $m\geq 1 $ such that $b_m a^{-1}=1$; a contradiction.

\end{proof}

\begin{rk} Theorem \ref{CarP+} does not hold in general when $G^+$ is not a $p$-group. Indeed, considering the holomorph of $\Z/5\Z$ (thus $G^+=\Z/5\Z \rtimes \Z/4\Z$), the group $G^+$ does not satisfy $\mathcal{P^+}$ (by Proposition~\ref{prop:3p3}, since $|G|< 24$). However its $2$-Sylow subgroups are cyclic, and it is easy to see that $G^+$ has at least two elements of order $2$, which are necessarily conjugate (since $2$-Sylow subgroups are conjugate). Thus $G^+$ contains an element of order $4$ whose square generates a non-normal subgroup.
\end{rk}

Along the lines of the remark above, we draw the following consequence of Theorem~\ref{CarP+} .

\begin{cor}
    Let $G^+$ be a $2$-group such that there exists $x\in G^+$ of order $4$ with $x^2 \notin \mathcal{Z}(G^+)$, then $G^+$ satisfies $\mathcal{P^+}$.
\end{cor}
\begin{proof}
Since $x^2$ has order $2$, $x^2 \notin \mathcal{Z}(G^+)$ implies that $\langle x^2 \rangle$ is not normal in $G^+$. Therefore we can apply Theorem~\ref{CarP+}.

\end{proof}

\begin{cor}\label{Matrices}
    Let $n \geq 3$ and $m\geq 2$ and let $G^+$ be the $2$-subgroup of $\GL_n(\Z/2^m\Z)$ consisting of all upper triangular matrices with diagonal coefficients all equal to $1$. Then $G^+$  satisfies $\mathcal{P^+}$; in particular the groups $\GL_n(\Z/2^m\Z)$ and $\mathrm{SL}_n(\Z/2^m\Z)$ satisfy $\mathcal{P^+}$.
\end{cor}
\begin{proof}
For $m\geq 2$, let $A=\Z/2^m\Z$. By Proposition \ref{Inj}, it is enough to show that $\GL_3(A)$ and $\mathrm{SL}_3(A)$ satisfy $\mathcal{P^+}$. Let $S$ be the subgroup of upper triangular matrices in $\mathrm{SL}_3(A)$ with diagonal coefficients all equal to $1$. 
Consider the following matrices
\[ M = \begin{pmatrix}
    \overline{1} & \overline{2}^{m-2} & 0 \\
    \overline{1} & \overline{1} & 0 \\
    0 & 0 & \overline{1} 
\end{pmatrix} \,,\qquad
N = \begin{pmatrix}
    \overline{1} & 0 & 0 \\
    0 & \overline{1} & \overline{1} \\
    0 & 0 & \overline{1} 
\end{pmatrix}\,.\]
We have $M^2\in S$ and $N\in S$, however $M^2 N \neq N M^2$ and therefore $M^2 \notin \mathcal{Z}(S)$. According to Corollary~\ref{Matrices}, $S$ satisfies $\mathcal{P^+}$, hence the result.

\end{proof}

Theorem \ref{p_groupsBias} is, thus, a direct consequence of theorem \ref{CarP+} and proposition \ref{P+simp}, whereas Corollary~\ref{p_groupsbiascor} is a consequence of Proposition~\ref{P+simp} and Corollary~\ref{Matrices}.

\subsection{Effective aspect}\label{subsec:linnik}

In this subsection we prove Theorem \ref{Linn1} using~\cite{FJ}*{Theorem 2.3}. Note that \cite{FJ}*{Theorem 2.3} in its generality assumes the Artin conjecture. In our case (where the groups considered are $p$-groups) Artin's conjecture is a proven result (see \cite{Bray}*{Chapter 2, Corollary 3.5}).

 \begin{proof}[Proof of Theorem~\ref{Linn1}]
 
Let $C_b$ denote the conjugacy class of $b$ in $G$, let \( s = \frac{|G|}{|C_b|}\mathds 1_{C_b} \) and \( s_\ell := s \circ f_\ell \) for \( \ell \geq 1 \).

For \( \ell \) coprime to \( p \), we have \( (t \circ f_\ell)^+ = t^+ \circ f_\ell = 0 \) by~\eqref{eq:indnulle},
and for \( \ell = pk \) a multiple of \( p \), we have \( (t \circ f_\ell) = -s_\ell = -(s_p \circ f_k) \). Based on divisibility properties of the indices appearing in the $\theta$ prime counting function (recall~\eqref{eq:phipsi}), we have:
\begin{align*} \theta(x;L/K;t) &= \sum_{\ell \geq 1} \mu(\ell) \psi\left(x^{\frac{1}{\ell}};L/K;t \circ f_\ell\right)
= \sum_{\ell \geq 1} \mu(\ell p) \psi\left(x^{\frac{1}{\ell p}};L/K;-(s_p \circ f_\ell)\right)\\
&= \sum_{\substack{\ell \geq 1 \\ \gcd(\ell,p) = 1}} \mu(\ell) \psi\left(x^{\frac{1}{\ell p}};L/K;s_p \circ f_\ell\right)\,.
 \end{align*}
Define
$u = \max \{ m \geq 1  \colon \exists x \in G,\, x^{p^m} \in C_b \} $ and, for \( 1 \leq k < u \), consider
\[ S_k := \sum_{1 \leq i \leq k} \theta(x^{\frac{1}{p^i}};L/K; s_{p^i} \circ f_i) - \theta(x;L/K;t)\,. \]
A simple induction on \( 1 \leq k < u \) shows that 
\[ S_k = -\sum_{\substack{\ell \geq 1 \\ \gcd(\ell,p) = 1}} \mu(\ell) \psi\left(x^{\frac{1}{\ell p^{k+1}}};L/K;s_{p^{k+1}} \circ f_\ell\right)\,. \]
In particular, for \( k = u-1 \), we have 
\[ S_{u-1} = -\sum_{\substack{\ell \geq 1 \\ \gcd(\ell,p) = 1}} \mu(\ell) \psi\left(x^{\frac{1}{\ell p^u}};L/K;s_{p^u} \circ f_\ell\right)\,. \]
If \( \ell = kp \) is a multiple of \( p \), then \( s_{p^u} \circ f_\ell = s_{p^{u+1}} \circ f_k \). Since \( \{x \in G \colon x^{p^{u+1}} \in C_b\} = \emptyset \) (by definition of \( u \)), we have \( s_{p^{u+1}} = 0 \)\,.

We conclude that 
\[ S_{u-1} = -\sum_{\ell \geq 1} \mu(\ell) \psi\left(x^{\frac{1}{\ell p^u}};L/K;s_{p^u} \circ f_\ell\right) = -\theta(x^{\frac{1}{p^u}};L/K;s_{p^u}) \,.\]
We have shown that 
\[ \theta(x;L/K;t) = \sum_{1 \leq k \leq u} \theta(x^{\frac{1}{p^k}};L/K;s_{p^k})\,, \]
from which we deduce by partial summation, 
\[ \pi(x;L/K;t) = \sum_{1 \leq k \leq u} \pi(x^{\frac{1}{p^k}};L/K;s_{p^k})\,. \]
This shows that \( \pi(x;L/K;t) \geq 0 \) for all \( x \). Moreover, we have \( \pi(x;L/K;t) > 0 \) as soon as \( \pi(x^{\frac{1}{p^k}};L/K;s_{p^k}) > 0 \) for some $k\in\{1,\ldots,u\}$.\\
Given a class function $g$ of $G$, denote, for all $\chi \in \text{Irr}(G)$, the fourier coefficient of $g$ at $\chi$ by \[\widehat{g}(\chi):=\langle \chi, g \rangle\, ,\]
and define the Littlewood norm of $g$ by \[\lambda(g):= \sum_{\chi\in \text{Irr}(G)}\chi(1)|\widehat{g}(\chi)|\, .\]
Since \( \widehat{s_p}(1) = r_p(C_b) > 0 \), we can apply~\cite{FJ}*{Th 2.3}: we choose 
\[ A = B \left( \frac{\lambda(s_p)}{\widehat{s_p}(1)} \log(rd_L+2)[K:\mathbb{Q}] \right)^{2p} \, ,\]
where \( B \) is a sufficiently large absolute constant. So that, for all \( x > A \), we have \( x^{\frac{1}{p}} > B \left( \frac{\lambda(s_p)}{\widehat{s_p}(1)} \log(\mathrm{rd}_L+2)[K:\mathbb{Q}] \right)^{2} \) and therefore according to \cite{FJ}*{Th 2.3} we have \
\[\pi(x^{\frac{1}{p}};L/K;s_p) > 0 \quad \text{ and } \quad \theta(x^{\frac{1}{p}};L/K;s_p) > 0 \,.
\]

Now, let \( D = \{x \in G \colon x^p \in C_b\} \); it is a conjugacy invariant set that we write \( D = \bigcup_{1 \leq i \leq r} F_i \), where the \( F_i \)'s are conjugacy classes of \( G \). We have
\[ \sum_{\chi} \chi(1) |\widehat{s_p}(\chi)| \leq \frac{1}{|C_b|} \sum_{1 \leq i \leq r} |F_i| \sum_{\chi} \chi(1) |\chi(F_i)| \leq \frac{|G|}{|C_b|} \sum_{i=1}^r \sqrt{|F_i|}\,. \]
Combining the Cauchy--Schwarz inequality and \( |D| = |C_b|r_p(C_b) \), we obtain
\[ \sum_{\chi} \chi(1) |\widehat{s_p}(\chi)| \leq \sqrt{r} |G| r_p(C_b)\,. \]
Since \( \widehat{s_p}(1) = r_p(C_b) \), we have 
\[ \frac{\lambda(s_p)}{\widehat{s_p}(1)} \leq \sqrt{r} |G|\,. \]
Therefore we can choose 
\[ A_1 = B(\sqrt{r}\log(\mathrm{rd}_L+2)[L:\mathbb{Q}])^{2p} \geq A\,. \]

\end{proof}


\section*{Appendix: proof of Theorem~\ref{Galois}}
Let $m\geq 2$, and define $\zeta:=\zeta_ {2^m}=\mathrm{e}^{(2i\pi)2^{-m}}$ and $K:=\Q(\zeta)$. Let $q\geq 3$ be a prime that is a sum of two squares: $q=a^2+b^2$ for some integers $a,b\geq 1$. We first note that since $q$ is ramified in $\Q(\sqrt{q})$ but unramified in $K$, then $\sqrt{q} \notin K$. We first prove the following Lemma.

\begin{lem}\label{Kummer}
    The polynomials $X^{2^m}-(a+ib)$ and $X^{2^m}-(a-ib)$ are irreducible in $K$. Moreover, if we let $D_1$ and $D_2$ be their respective splitting fields over $K$ then, $D_1 \cap D_2=K$ and $\Gal(D_1D_2/K)$ is isomorphic to $\ \mathbb{Z}/2^m\mathbb{Z} \times \mathbb{Z}/2^m \mathbb{Z}$.
\end{lem}

\begin{proof}
Since $a+ib$ and $a-ib$ are either both squares or both non-squares (since $K/\mathbb{Q}$ is a Galois extension), they cannot be squares, as otherwise $q=(a+ib)(a-ib)$ would be a square. Therefore, both $X^{2^m}-(a+ib)$ and $X^{2^m}-(a-ib)$ are irreducible in $K$ (see \cite{Algebra}*{Th. 6.9.1}).

By Kummer theory, $\Gal(D_1/K)$ and $\Gal(D_2/K)$ are both isomorphic to $\mathbb{Z}/2^m\mathbb{Z}$. In particular, $D_1,D_2$ are cyclic extensions of $K$. Then, by Galois theory, there exist unique intermediate extensions $K_1$ of $D_1/K$ and $K_2$ of $D_2/K$ such that $[K_1:K]=[K_2:K]=2$. Note that $K_1$ and $K_2$ are respectively generated by a square root of $a+ib$ and $a-ib$. 

Let $L=D_1 \cap D_2$ and assume by contradiction that $L\ne K$. Then $L$ is a common intermediate subextension of $D_1/K$ and $D_2/K$, and therefore $L$ is a cyclic extension of $K$ of degree $\geq 2$. Let $F$ be the unique intermediate extension of $L/K$ of degree $2$ over $K$. We have $K_1=F=K_2$. Let $\theta$ be a complex square root of $a+ib$ and thus $\overline{\theta}$ is a square root of $a-ib$. Since $K_1=K(\theta)$ and $K_2=K(\overline{\theta})$ and $[K_1K_2:K]=[K_1:K]=2$, we deduce that $\overline{\theta}\cdot\theta \in K$. Therefore, $(\overline{\theta}\cdot\theta)^2=\overline{\theta}^2\cdot\theta^2=(a-ib)(a+ib)=q$ is a square in $K$, which contradicts what was observed before the statement of Lemma~\ref{Kummer}. We conclude, as desired, that $L=K$ 
\end{proof}

\begin{proof}[Proof of Theorem~\ref{Galois}]

The polynomial
$P=(X^{2^m}-(a+ib))(X^{2^m}-(a-ib))(X^{2^m}-1)$ has coefficients
 in $\mathbb{Z}$, and all its roots are elements of $D_1D_2$. Moreover $D_1D_2$ is generated over $\mathbb{Q}$ by the roots of $P$ and therefore $D_1D_2$ is a splitting field of $P$ over $\mathbb{Q}$, making it Galois over $\mathbb{Q}$. Let $\pi \in \mathrm{Aut}(D_1D_2)$ be the restriction of complex conjugation to $D_1D_2$. We also denote $G_1=\Gal(D_1D_2/D_2)$ and $G_2=\Gal(D_1D_2/D_1)$.

Recall that restricting automorphisms to $D_1$ and $D_2$ induces two isomorphisms
\[
G_2\rightarrow \Gal(D_2/K)\,,\qquad G_1\rightarrow \Gal(D_1/K)\,.
\]
We have $\Gal(D_1D_2/K)=G_1G_2$ which is isomorphic (canonically) to $G_1\times G_2$. If $H$ is a non-trivial proper subgroup of $\ G_1$, then it is isomorphic to $\mathbb{Z}/2^n\mathbb{Z}$ for some $n\in\{1,\ldots,m-1\}$. Let $\theta_m \in D_1$ be a $2^m$-th root of $a+ib$, then $\overline{\theta_m}$ a $2^m$-th root of $a-ib$. Let $\gamma_1 \in G_1$ and $\gamma_2 \in G_2$ be such that $\gamma_1(\theta_m)=\zeta\theta_m$ and $\gamma_2(\overline{\theta_m})=\overline{\zeta \theta_m}$. Since $\zeta$ and $\overline{\zeta}$ are both $2^m$-th primitive roots of unity, $\gamma_1$ and $\gamma_2$ are generators of $G_1$ and $G_2$, respectively. Moreover $\pi \gamma_1 \pi (\overline{\theta_m})=\pi \gamma_1(\theta_m)=\pi(\zeta\theta_m)=\overline{\zeta\theta_m}$, and the restriction of $\pi \gamma_1 \pi $ to $D_1$ is the identity. We deduce that $\pi \gamma_1 \pi=\gamma_2$, showing that $\gamma_1$ and $\gamma_2$ are conjugate in $G^+$. Since $G_1$ is cyclic, $\gamma_1^{2^{m-n}}$ is a generator of $H$. Letting $\sigma_1=\gamma_1^{2^{m-n}}$ and $\sigma_2=\gamma_2^{2^{m-n}}$ we conclude that $G$ (resp. $G^+$) satisfies $\mathcal{P}$ (resp. $\mathcal{P^+}$) relatively to $(\sigma_1,\sigma_2)$.
\end{proof}

Keeping the notation as in Theorem~\ref{Galois}, we conclude by giving the structure of $\Gal(D_1D_2/\Q)$ which, by Proposition~\ref{Inj}, satisfies $\mathcal{P}^+$ relatively to $(\sigma_1,\sigma_2)$.

\begin{prop}
    We have an isomorphism $$\Gal(D_1D_2/\Q) \simeq ((\Z/2^m\Z \times \Z/2^m \Z)\rtimes \Z/2\Z)\rtimes \Z/2^{m-2}\Z\,.$$
\end{prop}
\begin{proof}
We keep the notation as in the proof of Theorem \ref{Galois}.

We have $\Gal(K/\Q)=\langle \pi, \tau_0 \rangle$ where $\tau_0(\zeta )=\zeta^5$. Let $\tau_1$ be an extension of $\tau_0$ to $D_1D_2$. The coset $\tau_1 \Gal(D_1D_2/K)$ contains exactly the automorphisms of $D_1D_2$ that restrict to $\tau_0$.

We claim that there exists $\tau \in G^+$ of order $2^{m-2}$ extending $\tau_0$. Indeed, as $\tau_0$ fixes $\Q(i)$ (recall that $\tau_0(i)=i$), $\tau_1$ permutes the roots of $X^{2^m}-(a+ib)$ (resp. $X^{2^m}-(a-ib)$). Moreover there exist $d_1,d_2 \in \Z$ such that $\tau_1(\theta_m)=\zeta^{d_1}\theta_m$ and $\tau_1(\overline{\theta_m})=\zeta^{d_2}\overline{\theta_m}$. Also since $5$ is invertible modulo $2^m$, there exist $r_1,r_2 \in \Z$ such that $5r_2+d_2 \equiv 5r_1+d_1 \equiv 0 \pmod{2^m}$. Set $\tau=\tau_1 \gamma_1 ^{r_1} \gamma_2 ^{-r_2}$. By construction, $\tau(\theta_m)=\theta_m$ and $\tau(\overline{\theta_m})=\overline{\theta_m}$. In particular we have $\tau^{2^{m-2}}=\Id_{D_1D_2}$. However $\Ord(\tau)\geq 2^{m-2}$ since $\tau$ extends $\tau_0$, and therefore $\Ord(\tau)=2^{m-2}$.

We have $\Gal(D_1D_2/\Q)=\langle \gamma_1, \gamma_2, \pi, \tau \rangle$, moreover, as $D_1D_2=\Q(\zeta,\theta_m,\overline{\theta_m})$ one easily checks the equalities
\[
\tau \pi \tau^{-1}=\pi\,,\qquad \tau \gamma_1 \tau^{-1}=\gamma_1^5\,,\qquad \tau \gamma_2 \tau^{-1}=\gamma_2^5\,.
\]
Therefore one has $G^+ \lhd \Gal(D_1D_2/\Q)$ and thus $\langle \gamma_1, \gamma_2, \pi, \tau \rangle=\langle \gamma_1, \gamma_2, \pi \rangle\cdot \langle \tau \rangle$. However, $G^+=\langle \gamma_1,\gamma_2, \pi \rangle \simeq (\Z/2^m\Z \times \Z/2^m \Z)\rtimes \Z/2\Z$. Therefore, for cardinality reasons (see Lemma \ref{card}), $\langle \gamma_1, \gamma_2, \pi \rangle \cap \langle \tau \rangle=\{1\}$. Then the product $\langle \gamma_1, \gamma_2, \pi \rangle\cdot \langle \tau \rangle$ is semi-direct which completes the proof.
\end{proof}

\begin{rk}
   In the case $m=2$, we recover an explicit solution to the inverse Galois problem for the group $(\Z/4\Z \times \Z/4\Z)\rtimes \Z/2\Z $ over $\Q$. Recall that by Proposition \ref{min-ab}, this is the smallest group containing $\Z/4\Z\times \Z/2\Z$ that satisfies $\mathcal{P^+}$.
    
\end{rk}

{\bf Acknowledgments.}
I would like to thank Florent Jouve for his constant support and patience, providing invaluable guidance and insights throughout the work that led to this paper. I would also like to thank Daniel Fiorilli for several discussions that significantly enriched my research.

\end{document}